\documentclass[a4paper,reqno]{amsart}

\usepackage{hyperref}

\usepackage{amssymb,amsxtra,amsfonts,dsfont}

\usepackage{graphicx} \usepackage{amsmath} \usepackage{amssymb}

\usepackage[svgnames]{xcolor}

\usepackage[color,matrix,arrow]{xy}

\input xy \xyoption{all}

\newtheorem{lem}{Lemma}[section]

\newtheorem{thm}[lem]{Theorem}

\newtheorem{pro}[lem]{Proposition}

\newtheorem{cor}[lem]{Corollary}

\newtheorem{exa}[lem]{Example}

\newcommand{\slrw}[1]{\stackrel{#1}{\longrightarrow}}

\newcommand{\ttn}[1]{\tau_n^{#1}}

\newcommand{\lrw}{\longrightarrow}

\newcommand{\LL}{\Lambda}

\newcommand{\xa}{\alpha}
\newcommand{\xb}{\beta}
\newcommand{\xc}{\gamma}

\newcommand{\caG}{\mathcal G}

\newcommand{\caD}{\mathcal D}
\newcommand{\caDb}{\mathcal D^b}
\newcommand{\caQ}{\mathcal Q}
\newcommand{\caM}{\mathcal M}
\newcommand{\caI}{\mathcal I}

\newcommand{\caP}{\mathcal P}

\newcommand{\RHom}{\mathbf{R}\strut\kern-.2em\operatorname{Hom}\nolimits}

\newcommand{\zZ}{{\mathbb Z}}

\newcommand{\zzv}{{\mathbb Z_{\vv}}}
\newcommand{\zzs}[1]{{\mathbb Z|_{#1} }}

\newcommand{\bai}{{\bar{i}}}
\newcommand{\baj}{{\bar{j}}}

\newcommand{\what}{\widehat}

\newcommand{\tLs}{\widetilde{\LL^{\sigma}}}
\newcommand{\tGs}{\widetilde{\GG^{\sigma}}}
\newcommand{\tL}{\widetilde{\LL}}

\newcommand{\dtL}{\Delta{\LL}}

\newcommand{\dtsL}{\Delta_{\sigma}{\LL}}
\newcommand{\dtnL}{\Delta_{\nu}{\LL}}

\newcommand{\olL}{\overline{\LL}}

\newcommand{\udgrm}{\underline{\mathrm{grmod}}}

\newcommand{\GG}{\Gamma}
\newcommand{\tG}{\widetilde{\Gamma}}
\newcommand{\trho}{\tilde{\rho}}

\newcommand{\olG}{\overline{\Gamma}}

\newcommand{\wht}[1]{\widehat{#1}}

\newcommand{\tQ}{\widetilde{Q}}

\newcommand{\olQ}{\overline{Q}}

\newcommand{\olrho}{\overline{\rho}}

\newcommand{\OO}{\Omega}

\newcommand{\identity}{\mathds{1}}

\newcommand{\Hom}{\mathrm{Hom}}

\newcommand{\Ext}{\mathrm{Ext}}

\newcommand{\soc}{\mathrm{soc}\,}

\newcommand{\vv}{\diamond}

\newcommand{\Ker}{\mathrm{Ker}\,}

\newcommand{\add}{\mathrm{add}\,}

\newcommand{\tr}{\mathrm{Tr}\,}

\newcommand{\mmod}{\mathrm{mod}\,}

\newcommand{\grm}{\mathrm{grmod}\,}

\newcommand{\caA}{{\mathcal A}}

\newcommand{\dmn}{\mathrm{dim}\,}

\newcommand{\dg}{\mathbf{deg}\,}

\newcommand{\bM}[1]{\mathbf{M}_{#1}\,}

\newcommand{\arr}[2]{\begin{array}{#1}#2\end{array}}

\newcommand{\eqqc}[2]{\begin{equation}\label{#1}#2\end{equation}}

\newcommand{\eqqcn}[2]{\[ #2 \]}
\newcommand{\eqqcnn}[2]{$ #2 $}

\newcommand{\mfd}{\mathfrak{d}}

\begin{document}

\title{On $n$-hereditary algebras and $n$-slice algebras}

\author[Guo and Hu]{Jin Yun Guo and Yanping Hu}
\address{Jin Yun Guo, Yanping Hu \\ MOE-LCSM, School of Mathematics and Statistics\\ Hunan Normal University\\ Changsha, Hunan 410081, P. R. China}

\email{gjy@hunnu.edu.cn(Guo), 1124320790@qq.com(Hu)}

\thanks{This work is partly supported by National Natural Science Foundation of China \#11671126, \#12071120 and the Construct Program of the Key Discipline in Hunan Province}

\begin{abstract}
In this paper we show that acyclic $n$-slice algebras are exactly
 acyclic $n$-hereditary algebras whose $(n+1)$-preprojective algebras are $(q+1,n+1)$-Koszul.
We also list the equivalent triangulated categories arising from the algebra constructions related to an $n$-slice algebra.
We show that higher slice algebras of finite type appear in  pairs and they share the Auslander-Reiten quiver for their higher preprojective components.
\end{abstract}

%\keywords{$n$-hereditary algebra; $(q+1,n+1)$-Koszul algebra; $n+1$-preprojective algebras; finite type; double slice}

\maketitle

\section{Introduction\label{int}}
The representation theory of hereditary algebras has been one of the central subjects in modern representation theory of algebras.
In recent years, $n$-hereditary algebras, including $n$-representation finite and $n$-representation infinite algebras, are introduced and studied as generalization of hereditary algebras in higher Auslander-Reiten theory (\cite{i007,i07,i11,io11,hio14}).
Graded self-injective algebras bear certain higher representation theory feature \cite{gw00}, its bound quiver is a stable $n$-translation quiver with natural $n$-translation $\tau$ induced by Nakayama functor \cite{g12,g16}.
An $n$-slice is the quadratic dual of a complete $\tau$-slice in such quiver, and  an $n$-slice algebra is an algebra whose bound quiver is an $n$-slice and whose  $(n+1)$-preprojective algebras is $(q+1,n+1)$-Koszul \cite{gxl21}.

In this paper, we first prove that acyclic $n$-hereditary algebras whose $(n+1)$-preprojective algebras are $(q+1,n+1)$-Koszul are exactly acyclic $n$-slice algebras (Theorem \ref{main} and Theorem \ref{main1}).
Recall that for the hereditary algebras, their preprojective algebras are $(q+1,2)$-Koszul algebras, so they are exactly $1$-slice algebras, and $q+3$ is exactly the Coxeter number of the underlying Dynkin diagram of its quiver when the algebra is of finite type \cite{bbk02}.
For an algebra defined by an $n$-translation quiver, being $(n+1,q+1)$-Koszul ensures its quadratic dual has $n$-almost split sequence in the category of  finitely generated projective modules under mild condition \cite{g16}.

The hereditary algebras are classified as of representation finite, tame and wild types.
For $n$-hereditary algebras, they are divided as  $n$-representation-finite and $n$-representation-infinite, and tame ones is defined inside $n$-representation-infinite \cite{hio14}.
A finite-tame-wild trichotomy for $n$-slice algebra using its Coxeter index and its spectral radius is proposed in \cite{gxl21}, which is consistent with \cite{hio14} for the finite-infinite dichotomy, but inconsistent for the tame case.
The tameness in \cite{hio14} is defined by the Noetherian $(n+1)$-preprojective algebra  and the tameness in \cite{gxl21} is defined by the finiteness of the Gelfand-Kirilov dimension of its $(n+1)$-preprojective algebra, or equivalently,  by the finite complexity of the trivial extension of its quadratic  dual.
It is  proved that the tameness defined in \cite{hio14} is the tameness defined in \cite{gxl21}, and the converse is conjectured to be true in \cite{gxl21}.

On the other hand, $n$-slice algebras of infinite type are  exactly the quasi $n$-Fano algebras  whose $(n+1)$-preprojective algebras are Koszul (or $(\infty,n+1)$-Koszul, see Theorem \ref{Fano}).

Apart from the  $(q+1,n+1)$-Koszulity of its  $(n+1)$-preprojective algebra, an $n$-slice algebra is defined using quiver with relations.
In fact, an $n$-slice is quadratic dual of $n$-properly-graded quiver $Q$  (see Sedtion \ref{pre} for definition), obtained by taking complete $\tau$-slice in the infinite stable $n$-translation quiver $\zzs{n-1} Q$(see Section \ref{eqtc} for definition), a generalization of $\zZ Q$.
So $n$-slice algebras provide a combinatoric approach to $n$-hereditary algebra, and this may leads to more direct generalization of the representation theory of hereditary algebras.
For example,  for an $n$-slice algebra with bound quiver $Q$, the Auslander-Reiten quiver of the $n$-preprojective and $n$-preinjective components are described using the quiver $\zzs{n-1} Q$ (see \cite{gll19b}), and the $n$-APR tilts are realized by a $\tau$-mutation of the $\tau$-slice in $\zzs{n-1} Q$ (see \cite{gx21}).
As a consequence, $n$-slice algebras and their types (finite, tame and wild) are invariant under APR tilts.
An $(n+1)$-slice algebra can be constructed from an $n$-slice using returning arrow constructions of such quivers, see \cite{gl16} for certain finite type ones and \cite{ghl22b} for infinite type.

Start from an $n$-slice algebra $\GG$ with bound quiver $Q^{\perp}$, write  $\LL=\GG^{!,op}$ for its quadratic dual.
We  revisit variants of constructions of quivers, such as the returning arrow quiver $\tQ$, the quivers $\zZ_{\vv} \tQ$ and $\zzs{n-1} Q$ related to $Q$,  and algebras related  to these quivers in Section \ref{eqtc}.
It is proved in \cite{g20} that certain twisted trivial extension $\tL$ of the quadratic dual $\LL$ of $\GG$ is isomorphic to the quadratic dual of the $(n+1)$-preprojective algebra $\tG=\Pi(\GG)$ of $\GG$ (See  \eqref{nslicedual} in Section \ref{pre}).
One also constructs the Beilinson-Green algebras $\LL^N$ of $\tL$ and $\GG^N$ of $\tG$, the smash products $\widehat{\LL^N}$ of $\tL$ and $\widehat{\GG^N}$ of $\tG$, and some other related algebras.
We listed in Section \ref{eqtc} the equivalent triangulate categories related to the derived category of an $n$-slice algebra arising from 
these algebras constructions.
For the $n$-slice algebra of infinite type, its $(n+1)$-preprojective algebra is an AS-regular algebra and this can be regarded as an algebra version of common generalization of Beilinson correspondence \cite{b78} and Bernstein-Gelfand-Gelfand correspondence \cite{bgg78} (see Picture \eqref{depictingcat} in Section \ref{eqtc}).
The quiver constructions are used in the last part of the paper to study the higher preprojective components of $n$-slice algebras of finite type.

Auslander-Reiten quivers are important tools in studying the representation theory of hereditary algebra of finite type.
By using hammocks  \cite{b86,rv87}, one obtains the Auslander-Reiten quiver of the path algebra $kQ$ of finite type from the quiver $\mathbb Z Q$ by attaching to each vertex of $Q$ the hammock starting at it.
This has a higher representation theoretic interpretation: higher slice algebras of finite type appear in pairs (Theorem \ref{fin}) and they share the same quiver as their Auslander-Reiten quiver of higher preprojective (higher preinjective) components(Theorem \ref{narq}).
We introduce double translation quiver and double slice to characterize these quivers: the Auslander-Reiten quiver of the $n$-preprojective component of  an $n$-slice algebra $\GG$ of finite type of Coxeter index $q+1$ is the opposite quiver of a double slice which is obtained by connecting  the quiver $Q^{\perp}$ of $\GG$  and the quiver ${Q^c}^{\perp}$ of its companion $\GG^c$,  a $q$-slice algebra of finite type of Coxeter index $n+1$.
Double slices provided higher representation theoretic analog of the constructions in   \cite{b86,rv87}.
In fact, double slice has been used in \cite{gl16} (called cuboid truncation there) to give a  iterated  construction of higher representation-finite algebra of type $A$.

The paper is organized as follow.
In Section \ref{pre}, we recall notions and some known results needed in the paper.
We prove that acyclic $n$-slice algebras are exactly acyclic $n$-hereditary algebras whose $(n+1)$-preprojective algebras are $(q+1,n+1)$-Koszul in the next two sections.
In Section  5, we recall algebra constructions related to an $n$-slice algebra and list the equivalent triangulate categories arising from these constructions.
In Section \ref{finite}, we prove that  higher slice algebras of finite type appear in  pairs, we also introduce double translation quivers and double slices and discuss their properties.
We show in the last section that the pair of the higher slice algebras of finite type share the Auslander-Reiten quiver for their higher preprojective components, with an example for illustratons.

\section{Preliminary\label{pre}}

In this paper, we assume that $k$ is a field, and $\LL = \LL_0 + \LL_1+\cdots$ is a graded algebra over $k$ with $\LL_0$ a direct sum of copies of $k$ such that $\LL$ is generated by $\LL_0$ and $\LL_1$.
Such algebra is determined by a bound quiver $Q= (Q_0,Q_1, \rho)$ \cite{g16}.

Recall that a bound quiver $Q= (Q_0,Q_1, \rho)$ is a quiver with $Q_0$ the set of vertices, $Q_1$ the set of arrows and $\rho$ a set of relations.
The arrow set $Q_1$ is usually defined with two maps $s, t$ from $Q_1$ to $Q_0$ to assign an arrow $\alpha$ its starting vertex $s(\alpha)$ and its ending vertex $t(\alpha)$.
Write $Q_t$ for the set of paths of length $t$ in the quiver $Q$, and write $kQ_t$ for space spanned by $Q_t$.
Let $kQ =\bigoplus_{t\ge 0} kQ_t$ be the path algebra of the quiver $Q$ over $k$.
We also write  $s(p)$ for the starting vertex of a path $p$ and $t(p)$ for the terminating vertex of $p$.
Write $s(x)=i$ if $x$ is a linear combination of paths starting at vertex $i$, and  write $t(x)=j$ if $x$ is a linear combination of paths ending at vertex $j$.
The relation set $\rho$ is a set of linear combinations of paths of length $\ge 2$. We may assume that the paths appearing in each of the element in $\rho$ have the same length, since we deal with graded algebra.
Through this paper, we assume that the relations are normalized such that each element in $\rho$ is a linear combination of paths of the same length starting at the same vertex and ending at the same vertex.
Conventionally, modules are assumed to be finitely generated left module in this paper.

Let $\LL_0=  \bigoplus\limits_{i\in Q_0} k_i$, with $k_i \simeq k$ as algebras, and let $e_i$ be the image of the identity in the $i$th copy of $k$ under the canonical embedding.
Then $\{e_i| i \in Q_0\}$ is a complete set of orthogonal primitive idempotents in $\LL_0$ and $\LL_1 = \bigoplus\limits_{i,j \in Q_0 }e_j \LL_1 e_i$ as $\LL_0 $-bimodules.
Fix a basis $Q_1^{ij}$ of $e_j \LL_1 e_i$ for any pair $i, j\in Q_0$, take the elements of $Q_1^{ij}$ as arrows from $i$ to $j$, and let $Q_1= \cup_{(i,j)\in Q_0\times Q_0} Q_1^{ ij}.$
Thus $\LL \simeq kQ/(\rho)$ as algebras for some relation set $\rho$, and $\LL_t \simeq kQ_t /((\rho)\cap kQ_t) $ as $\LL_0$-bimodules.
A path whose image is nonzero in $\LL $ is called a {\em bound path}.

\medskip

A bound quiver $Q= (Q_0,Q_1,\rho)$ is called  {\em quadratic} if $\rho$ is a set of linear combination of paths of length $2$.
In this case, $\LL= kQ/(\rho)$ is called a {\em quadratic algebra}.

Regard the idempotents $e_i$ as the linear functions on $\bigoplus_{i\in Q_0} k e_i$ such that $e_i(e_j) = \delta_{i,j}$, and for $i,j\in Q_0$, arrows from $i$ to $j$ are also regarded as the linear functions on $e_j k Q_1 e_i$ such that for any arrows $\xa, \xb$, we have $\xa(\xb) =\delta_{\xa,\xb}$.
By defining $\xa_1 \cdots \xa_r (\xb_1 \cdots \xb_r)=\xa_1(\xb_1) \cdots \xa_r ( \xb_r) $, $e_j kQ_r e_i$ is identified with its dual space for each $r$ and each pair $i,j$ of vertices, and the set of paths of length $r$ is regarded as the dual basis to itself in $e_j kQ_r e_i$.
Take a spanning set $ \rho^{\perp}_{i,j}$ of orthogonal subspace of the set $e_j k\rho e_i$ in the space $e_j kQ_2 e_i$ for each pair $i,j \in Q_0$, and set \eqqc{relationqd}{\rho^{\perp} = \bigcup_{i,j\in Q_0} \rho^{\perp}_{i,j}.}
The {\em quadratic dual quiver} of $Q$ is defined as the bound quiver $Q^{\perp} =(Q_0,Q_1, \rho^{\perp})$, and the {\em quadratic dual algebra} of $\LL$ is the algebra $\LL^{!,op} \simeq kQ/(\rho^{\perp})$ defined by the bound quiver $Q^{\perp}$ (see \cite{g20}).

\medskip

Recall that a bound quiver $Q =(Q_0,Q_1,\rho)$  is called {\em $n$-properly-graded} if all the maximal bound paths have the same length $n$.
The graded algebra $\LL = k Q/(\rho)$ defined by an  $n$-properly-graded quiver is called an {\em $n$-properly-graded algebra.}

If an  $n$-properly-graded quiver $Q$ is quadratic, its quadratic dual $Q^{\perp}$ is called an {\em $n$-slice}.

\medskip

Let $Q = (Q_0,Q_1,\rho)$ be an acyclic bound quiver without bound path of  infinite length. 
let $\LL$ be the algebra defined by $Q$.
The returning arrow quiver $\tQ$ is the quiver with the same vertex set $\tQ_0=Q_0$ as $Q$ and arrow set $\tQ_1 = Q_1 \cup Q_{1,\caM}$, with  $Q_{1,\caM}=\{\xb_p:t(p)\rightarrow s(p) | p\in \caM\}$, where $\caM$ is a basis for $\soc_{\LL^e} \LL$. 
If $Q$ is an $n$-properly-graded quiver,  $\caM$ is a maximal linearly independent set of maximal bound paths of $Q$ and the  relation set $\trho$ will be discussed in this case.

By \cite{g20}, there is an epimorphism $\phi$ from the free $\LL$-bimodule $M_{\caM}$ generated by $Q_{1,\caM}$ to $D\LL$ if $Q$ is $n$-properly-graded quiver, take $\rho_{\caM}$ a generating set of $\ker \phi$.

Recall that the trivial extension $\LL\ltimes M$ of an algebra $\LL$ by an $\LL$-bimodule $M$ is the algebra defined on the vector spaces $\LL\oplus M$ with the multiplication defined by $(a,x)(b,y)=(ab, ay+xb)$ for $a,b\in \LL$ and $x,y\in M$.
$\dtL = \LL\ltimes D\LL$ is called the  trivial extension of $\LL$.
Let $\sigma$ be a graded automorphism of $\LL$, that is, an automorphism that preserves the degree of homogeneous element.
Let $M$ be an $\LL$-bimodule.
Define the twist $M^{\sigma}$ of $M$ as the bimodule with $M$ as the vector space.
The left multiplication is the same as $M$, and the right multiplication is twisted by $\sigma$, that is,  defined by $x\cdot b=x\sigma(b)$ for all $x \in M^{\sigma}$ and $b \in \LL$.
Define the twisted trivial extension $ \dtsL= \LL \ltimes D \LL^{\sigma}$ to be the trivial extension of $\LL$ by the twisted $\LL$-bimodule $D \LL^{\sigma}$ if $Q$ is $n$-properly-graded quiver.
Write $\rho_{\caM, \sigma} = \{\sum_{}x_{p,t}\xb_{p,t}\sigma(y_{p,t})|\sum_{}x_{p,t}\xb_{p,t} y_{p,t} \in \rho_{\caM} \}$.
Then $\rho_{\caM, \sigma} $ is a relation set for the bimodule $ D \LL^{\sigma}$  if $Q$ is $n$-properly-graded quiver. 
 Write $Q_{2,\caM}$ for the set of paths of length $2$ formed by arrows in  $Q_{1,\caM}$.
By Proposition 2,2 of \cite{fp02} and Theorem 3.1 and Corollary 3.3 of \cite{g20}, we have the following results. 

\begin{pro}\label{returning_trivial_ext}
If $\LL$ is  an algebra with acyclic bound quiver $Q$, then $\dtL\simeq k\tQ/(\trho)$ and $\dtsL\simeq k\tQ/(\trho^{\sigma})$.

 If $Q$ is an $n$-properly-graded quiver, then $$\trho = \rho\cup\rho_{\caM}\cup Q_{2,\caM}  \mbox{ and }\trho^{\sigma} = \rho\cup\rho_{\caM, \sigma}\cup Q_{2,\caM}.$$ 
\end{pro}
Note that $\tQ$ is a stable $n$-translation quiver introduced in \cite{g12,g16}.
The bound quivers of  graded self-injective algebras  are exactly {\em the stable $n$-translation quivers}, with the Nakayama permutation $\tau$ as its $n$-translation \cite{g16} (called stable bound quiver of Loewy length $n+2$ in \cite{g12}).

Write $\Pi(\GG)$ for the $(n+1)$-preprojective algebra of $\GG$, the following result is proved in \cite{g20}.

\begin{thm}\label{mainthm}
Assume that $\LL$ is a Koszul $n$-properly-graded  algebra and $\dtL$ is quadratic. Then $$\Pi( \LL^{ !, op }) \simeq (\dtnL)^{!,op},$$
for the graded automorphism $\nu$ of $\LL$ which sending an arrow $\xa$ to $(-1)^{n} \xa$ .
\end{thm}

So we have that $\Pi( \LL^{ !, op })\simeq k\tQ/(\trho^{\nu,\perp})$.

\medskip

Recall that a graded algebra $\tL= \sum_{t\ge 0} \tL_t$ is called a {\em $(p,q)$-Koszul algebra} if $\tL_t=0$ for $t> p$ and $\tL_0$ has a graded projective resolution \eqqc{projres_tL}{\cdots \lrw P^{t} \slrw{f_t}  \cdots \lrw P^{1}\slrw{f_1} P^{0} \slrw{f_0} \tL_0\lrw 0,} such that $P^{t}$ is generated by its degree $t$ part for $t\le q$ and $\ker f_q$ is concentrated in degree $q+p$ \cite{bbk02,g16}.
The concept of $(p,q)$-Koszul unifies almost Koszul and Koszul by allowing $p,q$ to be infinite: a $(p,q)$-Koszul algebra is Koszul when one of $p, q$ is infinite\cite{bbk02}.
Note that a $(p,q)$-Koszul algebra is quadratic.

\medskip

A graded algebra $\tL$ defined by an $n$-translation quiver $\tQ$ is called an {\em $n$-translation algebra}, if there is a $q \ge 1$ or $q=\infty$ such that $\tL$ is an $(n+1,q+1)$-Koszul algebra \cite{g16}.
Conventionally, we take $q+1 = \infty$ when $q=\infty$.
A {\em stable $n$-translation algebra} $\tL$ is a $(n+1,q+1)$-Koszul self-injective algebra for some $q  \ge 1$ or $q=\infty$, and we call $q+1$ the {\em Coxeter index} of $\tL$.

\medskip

Let $\LL$ be the $n$-properly-graded algebra defined by $Q$, if the trivial extension $\dtL$ is a stable $n$-translation algebra, the quadratic dual algebra $\GG=kQ/(\rho^\perp)$ defined by the $n$-slice $Q^{\perp}$ is called an {\em $n$-slice algebra}.

Theorem \ref{mainthm} tells us  that \eqqc{nslicedual}{\Pi( \GG) \simeq( \Delta_{\nu} (\GG^{!,op}))^{!,op}} for an $n$-slice algebra $\GG$.

\medskip
The classification of $n$-slice algebra are discussed in \cite{gxl21}, where the Loewy matrix of its quadratic dual plays an important role.

Let $\tG= \Pi(\GG)$ be the $(n+1)$-preprojective algebra of an $n$-slice algebra $\GG$.
Recall that $\GG$ is  of {\em finite type} if $\tG$ is finite dimensional, of {\em tame type} if  $\tG$ is of finite Gelfand-Kirilov dimension (not zero), and of {\em wild type}  if $\tG$ is of infinite Gelfand-Kirilov dimension.
Let $\tL$ be the quadratic dual of $\tG$.
This trichotomy of $n$-slice algebra $\GG$ is also characterized by using the Coxeter index and the Loewy matrix of $\tL$,  see  Theorem 3.7 of \cite{gxl21}.

For an $n$-slice algebra $\GG$, its quadratic dual $\LL$ is an $n$-properly-graded algebra.
Let $\tL$ be its twisted trivial extension, it is a $(n+1,q+1)$-Koszul selfinjective algebra of Loewy length $n+2$.
Write $L= L_{\tL}$ for the Loewy matrix of $\tL$ \cite{gw00} , we have following characterization of finite-tame-wild trichonomy using $\tL$(Theorem 3.7 of \cite{gxl21}).

\begin{thm}\label{trichotomy}
Let $\GG$ be an $n$-slice algebra with Coxeter index $q+1$.
\begin{enumerate}
 \item $\GG$ is of finite type if and only if   $q+1$ is finite.

 \item $\GG$ is of tame type  if and only if $\tL$ is of finite complexity and not periodic, if and only if $q$ is infinite and  the  spectral radius of $L$ is $1$.

 \item $\GG$ is of wild type if and only if $\tL$ is of infinite complexity, if and only if  $q$ is infinite and  the  spectral radius of $L$ is larger than $1$.
 \end{enumerate}
\end{thm}

\medskip

\section{Twisted trivial extensions and higher preprojective algebras \label{nhns}}

Let $\GG$ be a finite dimensional algebra of global dimension $\le n$, let $\tG =\Pi(\GG)$ be the $(n+1)$-preprojective algebra of $\GG$. 
 If $\GG$ is quadratic, write $\LL= \GG^{!,op} $ for its quadratic dual.  

Assume that $\LL$ is an acyclic n-properly-graded algebra, consider the twisted trivial extension $\dtsL$ of $\LL$ with respect to an automorphism $\sigma$.
With the returning arrow as degree $1$ element, $\dtsL$ are naturally a graded algebras extending the gradation of $\LL$ as its first gradation.
With the algebra $\LL$ as degree $0$ component and $D\LL^{\sigma}$ as degree $1$ component, $\dtsL$ is endowed with a second gradation and is made a bigraded algebra graded by $\zZ^2=\zZ \times \zZ$.
This is equivalent to take the gradation of the tensor algebra $T_{\GG}(\Ext_{\GG}^n(D \GG, \GG))$ as the second gradation,  and $\tG$ is a bigraded algebra.

As the algebra defined by the returning arrow quiver $\tQ^{\sigma}$, the second gradation is defined by endowing the old arrows in $Q$ with degree $0$ and the returning arrows with degree $1$.
We have the decomposition $\dtsL= \sum_t\dtsL_{t,-}$ with respect to the first gradation and $\dtsL=\sum_t\dtsL_{-,t}$ with respect to the second gradation.
Write $()_{*,-}$ for the shift with respect to the first gradation and
$()_{-,*}$ for the shift with respect to the second gradation.

We have that \eqqc{fdegtL}{\dtsL_{t,-} = \LL_t + D\LL_{n+1-t},} with the convention that $D\LL_{n+1}=0$.
So \eqqc{dualbimod}{D\LL_{1,-} =  kQ_{1,\caM} =\sum_{p\in \caM} k\xb_p = D\LL_{n},} and  \eqqc{fdegDL}{ D\LL_{t,-} = \sum_{s=0}^n \LL_{s}D\LL_{n}\LL_{t-1-s} \mbox{ and } D\LL_{-,1} = D\LL,} that is $D\LL$ is concentrated in degree $(-,1)$ with respect to the second grading.
Especially
\eqqc{degree1fortL}{ \dtsL_{-,0} =  \LL \mbox{ and } \dtsL_{-,1} =D\LL.}

Now assume that $\tLs= \dtsL$ is quadratic, write $\tGs= \tL^{\sigma,!,op}$ for its quadratic dual.
Then since $\tGs \simeq k\tQ/(\trho^{\sigma,\perp})$, the second gradation of $\tLs$ induces a second gradation on $\tGs$, by endowing the arrows in $Q_1$ with degree $0$ and the arrow in $Q_{1,\caM}$ with degree $1$.
In this case, we obviously have that
\eqqc{degreezerofortG}{\tGs_{-,0} =  \GG.}

Assume that  $\tG $ is a $(q+1,n+1)$-Koszul algebra with respect to the first gradation, it is quadratic with this gradation.
The bigradation of $\tG$ is coherent with the  bigradation on its quadratic dual  $\tL=\tG^{!,op}$, and $\tL$ is an  $(n+1,q+1)$-Koszul algebra with respect to the first gradation.
So we have a Koszul complex \eqqc{Koszulc}{ \tG\otimes D\tL_{(n+1,*)} \slrw{f_{n+1}} \tG\otimes D\tL_{(n,*)} \lrw \cdots \to \tG\otimes D\tL_{(1,*)} \slrw{f_1} \tG\otimes D\tL_{(0,*)}} with respect to the first gradation.

\begin{lem}\label{GGKoszul}
If $\tG$ is $(q+1,n+1)$-Koszul, then $\GG$ is Koszul.
\end{lem}
\begin{proof}
Since $\tG$ is $(q+1,n+1)$-Koszul, we have that \eqref{Koszulc} is exact except for the first and the last term.
The cokernel of $f_1$ is $\tG_0=\GG_0$,  kernel of $f_{n+1}$ is concentrated at degree $(n+1+q,*)$ if $q$ is finite, and is $0$ otherwise.
By taking the components with the second degree $0$, we get \small\eqqc{KoszulcGG}{\tG\otimes D\tL_{(n+1,0)} \slrw{f_{n+1}} \tG\otimes D\tL_{(n,0)} \lrw \cdots \to \tG\otimes D\tL_{(1,0)} \slrw{f_1} \tG\otimes D\tL_{(0,0)}\lrw \GG_0\lrw 0.
}\normalsize
Note that $\tL_{(n+1,0)} =0$, and $\tG\otimes D\tL_{(t,0)} = \GG \otimes D\tL_{(t,0)}$ is generated in degree $(t,0)$ for $0\le t \le n$, and $\tG_{(0,0)} = \GG_0$.
So we get an exact sequence \eqqc{KoszulcGG}{
0 \lrw \GG\otimes D\tL_{(n,0)} \lrw \cdots \to \GG\otimes D\tL_{(1,0)} \lrw \GG\otimes D\tL_{(0,0)}\lrw \GG_0\lrw 0,}
which is a projective resolution of $\GG_0$ with the $t$th term $\GG \otimes D\tL_{(t,0)}$ generated in degree $t$, so $\GG$ is Koszul.
\end{proof}

\begin{pro}\label{kknslice}
Let $\LL$ be an acyclic finite dimensional Koszul algebra.
If its twisted trivial extension  $\tL$ is $(n+1,q+1)$-Koszul, then  $\LL$ is an $n$-properly-graded algebra.
\end{pro}

\begin{proof}
Let $Q$ be the bound quiver of $\LL$ and let $\tQ$ be the bound quiver of $\tL$. 
Then $\tQ$ is the returning arrow quiver of $Q$.

Now $\tL$ is $(n+1,q+1)$-Koszul selfinjective algebra, projective injective $\tL$-modules have the same Loewy length $n+2$, so $\tQ$ is a stable $n$-translation quiver  with identity as $n$-translation, hence is $(n+1)$-properly graded, and $\tL$ is an $n$-translation algebra. 

Let $p = \xa_r \cdots \xa_1$ be a bound path in $Q$ of maximal length with arrows $\xa_t: i_{t-1}\to i_t$ and $i=i_0, j=i_r$, it is a bound path in $\tQ$.
So there is a bound path $p'= \xc_{n+1-r}\cdots\xc_{1}$ such that $p'p$ is a bound path in $\tQ$ from $ j$ to $j$.

Note that we have that $(D\LL)^2 =0$, we have only one returning arrow in  $p'$, say $\xc_s =\xb_q$ in $p'$.
If $r<n$, then $p' = p_1\xb_q p_2$ for some path $p_1,p_2 $ in $Q$ and now we have that $0 \neq p'p = p_1\xb_q p_2 p $, so $\xb_q p_2 p p_1 \neq 0$ and $p_2p p_1 \neq 0$, this contradicts the choice of $p$.

So $r=n$ and $Q$ is $n$-properly-graded.  
\end{proof}

\begin{thm}\label{main}
Let $\GG$ be an acyclic $n$-hereditary algebra. If its $(n+1)$-preprojective algebra $\tG= \Pi(\GG)$ is $(q+1,n+1)$-Koszul, then  $\GG$ is an  $n$-slice algebra.
\end{thm}
\begin{proof}
Let $\GG$ be an $n$-hereditary algebra with bound quiver $Q^{\perp}$.
Let $\tG= \Pi(\GG)$ be the $(n+1)$-preprojective algebra of $\GG$.

Let $\LL = \GG^{!,op}$ with bound quiver $Q$, then by Theorem C %5.2 and 5.4 
of \cite{gi20}, we have the isomorphism $\tG^{!,op}\simeq \tL^{\sigma}$ of for some twisted trivial extension $\tL^{\sigma}$ of $\LL$.
Since $\tG$ is  $(q+1,n+1)$-Koszul, $\tL^{\sigma}$ is $(n+1,q+1)$-Koszul,  so  it is an stable $n$-translation algebra.
Thus $\tL$ is also an  stable $n$-translation algebra.
By Proposition \ref{kknslice}, $Q$ is $n$-properly-graded quiver, and $Q^{\perp}$ is an $n$-slice.

This prove that $\GG$ is an $n$-slice algebra.
\end{proof}

\section{$n$-hereditary algebras and $n$-slice algebras\label{ncyp}}

In this section we prove that an $n$-slice algebra is $n$-hereditary by showing that its $(n+1)$-preprojective algebra is $(n+1)$-Calabi-Yau or stably $(n+1)$-Calabi-Yau.

We first study the Nakayama automorphism twisted trivial extension of an $n$-properly-graded algebra.

Let $\LL$ be an $n$-properly-graded algebra with acyclic bound quiver $Q$.
By Proposition \ref{returning_trivial_ext}, the bound quiver of the trivial extension $\dtL$ of $\LL$ is the returning arrow quiver $\tQ$ of $Q$.
Let $\{p^*|p\in \caM\}$ be the dual basis of $\caM$ in $D\LL_n$.
Note that for each bound path $q=\xa_r\cdots\xa_1$ in $\dtL$, write $$\zeta_q =\left\{ \arr{ll}{0& \xa_t \in Q_1, 1\le t\le r \\ p^* & \xa_{r_0} = \xb_{p} \in Q_{1,\caM}, 1\le r_0\le r } \right.$$ the map
\eqqc{lmap}{\mu: q \to \zeta_q(\xa_{r_0-1}\cdots\xa_1\xa_{r}\cdots \xa_{r_0+1})}
defines a linear map on $\dtL$.
Set \eqqc{bform}{(x,y) = \mu(xy),} this is a non-degenerate bilinear form on $\dtL$.
The algebra $\dtL$ is symmetric, so the Nakayama automorphism $\omega$ of $\dtL$ is the identity, and we have that \eqqc{sym}{(x,y)=(y,x).}

Now we consider the Nakayama automorphism of $\dtsL$, that is, the automorphism $\omega$ of $\dtsL$, such that $(a,b) = (b, \omega(a))$, with respect to the non-degenerate bilinear form above.
Since $\sigma$ is an automorphism on $\LL$, the map $\omega$ sending $p^*$ to $p^*\sigma$ defines an automorphism on the vector space $D\LL_n$, which is isomorphic to the subspace $kQ_{1,\caM}$.

\begin{pro}\label{nakayama_auto}
Let $\LL$ be an acyclic algebra and let $\sigma$ be a graded automorphism of $\LL$.
Then the Nakayama automorphism of $\dtsL$  is the automorphism defined by $\omega(\xb_p)= p^*\sigma$ on $Q_{1,\caM}$, $\omega(e_i) =e_i$ for $i\in Q_0$ and  $\omega=\sigma^{-1}$ on $Q_1$.
\end{pro}
\begin{proof}
If $q=e_i$ is a primitive idempotent and $q'$ be a bound path in $\tQ$, clearly we have that $(q,q')\neq 0$ if $q'$ is a cyclic bound path in $\tQ$ starting and ending at the vertex $i$.
Thus  we have $(e_i,q')= (q',e_i)$ and $\omega(e_i)= e_i$.

Now assume that $q$ is an arrow,  and let $q'$ is a bound path in $Q$.

If $q= \xb_p$ is a returning arrow, then $$(q,q') =\mu( \xb_p\cdot q') =\mu( \xb_p\sigma( q')) = p^*(\sigma(q')) = (p^*\sigma)(q'), $$ while $$(q',q) =\mu( q' \xb_p) = p^*(q') . $$
Since $\omega(\xb_p)$ is the isomorphic image of $p^*\sigma$ in $kQ_{1,\caM}$,  $$(q',\omega(q) )=\mu( q'\omega(\xb_p)) = (p^*\sigma)(q').$$

Assume that $q=\xa$ is an arrow in $Q$.
If $q'$ is in $Q$, then $(q,q')=0=(q',q)$.
Now let $q' = u\xb_p v$, then \eqqcn{cc}{\arr{lll}{(q,q') & = & \mu(\xa u\xb_p v )  = p^*(v\xa u) }} and \eqqcn{cc}{\arr{lll}{(q',\sigma^{-1}(q)) & = & \mu( u\xb_p v\cdot \sigma^{-1}(\xa) ) =\mu( u\xb_p v\sigma( \sigma^{-1}(\xa) ) )= p^*(v\xa u) } =(q, q').}
Define \eqqcnn{}{\omega(\xa) =\sigma^{-1} (\xa)} for $\xa \in Q_1$, then $\omega$ defines an automorphism of the space $k\tQ_1$.

Now extend $\omega$ to an algebra automorphism on $\dtsL =k\tQ/(\trho^{\sigma})$.
For bound paths $u,u',v$ in $Q$ and arrow $\xb_p$ in $Q_{1,\caM}$,  we have that $$\arr{lll}{(u', u\xb_pv) &=& \mu (u' u\xb_pv) = p^*(vu'u) = \mu(u\xb_pv \sigma(\sigma^{-1} (u'))) =( u\xb_pv ,\sigma^{-1} (u')) \\ &=&( u\xb_pv ,\omega(u')),  \\ (u\xb_pv,u') &=& \mu (u\xb_pv\sigma(u')) = p^*(v\sigma(u')u) = p^*(\sigma (\sigma^{-1} (v) u'\sigma^{-1} (u)))\\&=& \mu(u'\omega (u)\omega(\xb_p)\omega(v)) \\ &=&(u', \omega(u\xb_pv)) .} $$
So we get that $\omega$ is the Nakayama automorphism of $\dtsL$.
\end{proof}

Calabi-Yau category is introduced in \cite{ko98} (see \cite{ke99}).
Let $\mathcal T$ be a $k$-linear triangulated category with finite dimensional Hom spaces. The category $\mathcal T$ is said to be $d$-Calabi-Yau if for all objects $X,Y$ in $\mathcal T$ , there exists an isomorphism $\Hom_{\mathcal T}  (X,Y ) \simeq D\Hom_{\mathcal T} (Y,X[d])$ functorial in $X$ and $Y$.

Recall that a graded algebra $\tG$ is called an bimodule $(d + 1)$-Calabi-Yau algebra of Gorenstein parameter $p$ if it is an homologically smooth algebra satisfying
$$\RHom_{\tG^e} (\tG,\tG^e)[d + 1] \simeq \tG(p)$$ in $\caD^b(\tG^e)$.

A graded algebra $\tG$ is called bimodule {\em stably $d$-Calabi-Yau of Gorenstein parameter $p$} if there is an isomorphism $$\RHom_{\tG^e} (\tG,\tG^e)[d + 1] \simeq \tG(p)$$ in $\udgrm_{CM}(\tG^e)$,
 where $\udgrm_{CM}(\tG^e)$ is the stable category of graded Cohen-Macaulay $\tG^e$-modules.

\begin{thm}\label{main1}
An  acyclic $n$-slice algebra  is an $n$-hereditary algebra.
\end{thm}
\begin{proof}
Let $\GG$ be an $n$-slice algebra and let $\LL$ be its quadratic dual.
Then $\LL$ is an $n$-properly-graded algebra whose twisted trivial extension $\dtsL$ are stable $n$-translation algebra for all graded automorphism $\sigma$ of $\LL$.

So there is an $q$ such that $\dtsL$ is an $(n+1,q+1)$-Koszul self-injective algebra.
Let $\epsilon$ be the automorphism of $\LL$ sending arrows to its scalar by $(-1)$.
Take $\sigma =\epsilon^{n}$, it is the automorphism of $\LL$ sending arrows to its scalar by $(-1)^{n}$, then by Theorem 5.3. of \cite{g20}, $\Pi(\GG) = \dtsL^{!,op}$.

We have $\omega(\xb_p)x = p^*\epsilon^{n}\mbox{ }(x) =(-1)^{n}p^*(x) = (-1)^{n}\xb_p x $, so $\omega(\xb_p)= (-1)^{n}\xb_p$.
This shows that the Nakayama automorphism $\omega$ of $\Delta_{\sigma}\LL$ maps $x$ to $(-1)^{n\dg_1(x)}x$

If $q$ is infinite, then $\dtsL$ is Koszul of with  $(n+1)$th power of its radical zero, and $\tG= \Pi(\GG)$ is of global dimension $n+1$.
Since $(-1)^n=(-1)^{n+2}$, by Theorem 4.2 of \cite{rrz17}, the quadratic dual $\tG  = \dtsL^{!,op}$ is  Calabi-Yau.

By \eqref{degreezerofortG}, we have $\tGs_{-,0} =  \GG$ so \eqqc{}{\RHom_{\tG^e} (\tG,\tG^e)[n + 1] \simeq \tG(1)_{-,*}} in $\caD^b(\tG^e)$, and $\GG$ is $n$-representation-infinite, by Theorem 3.1 of \cite{a14}.

If $q$ is finite, then $\dtsL$ is an$(n+1,q+1)$-Koszul self-injective algebra, it is Frobenius of periodic type.
So by Theorem 3.4 and Corollary 4.2 of \cite{yu12}, $\tG = \dtsL^{!,op}$ is stably Calabi-Yau.
By \eqref{degreezerofortG}, we have $\tGs_{-,0} =  \GG$ so \eqqc{}{\RHom_{\tG^e} (\tG,\tG^e)[n + 1] \simeq \tG(1)_{-,*}} in $\udgrm_{CM}(\tG^e)$, and $\GG$ is $n$-representation-finite, by Theorem 3.2 of \cite{a14}.

This proves that an acyclic $n$-slice algebra  $\GG$ is $n$-hereditary algebra and $\tGs$ is its $(n+1)$-preprojective algebra for $\sigma =\epsilon^{n+2}$.
\end{proof}

Combine Theorem \ref{main} and \ref{main1}, we prove that
\begin{thm}\label{suf_nec}
An acyclic $n$-hereditary algebra is $n$-slice algebra if and only if its $(n+1)$-preprojective algebra is $(q+1,n+1)$-Koszul for some $q$.
\end{thm}

In fact, $n$-slice  algebras of infinite type are exactly the $n$-Fano  algebras with Koszul $(n+1)$-preprojective algebras, as is proved in the next theorem.

\begin{thm}\label{Fano}
An acyclic algebra $\GG$ is  quasi $n$-Fano algebra with Koszul $(n+1)$-preprojective algebra if and only if it is an $n$-slice algebra of infinite type.
\end{thm}
\begin{proof}
By Theorem \ref{main}, we need only to prove that quasi $n$-Fano algebra with Koszul $(n+1)$-preprojective algebra is an $n$-slice algebra of infinite type.

If $\GG$ is  quasi $n$-Fano algebra, then by Theorem 4.2 of \cite{mm11}, its  $(n+1)$-preprojective algebra $\tG$ is AS-regular of dimension $n+1$.
Since $\tG$ is Koszul, by Theorem 5.1 of \cite{m99}, its quadratic dual $\tL= \tG^{!,op}$ is Koszul self-injective algebra of Loewy length $n+1$, that is, $\tL$ is a stable $n$-translation algebra.
By Proposition \ref{kknslice}, $\GG$ is an $n$-slice algebra.
By Theorem 3.2 of \cite{gxl21}, $\GG$ is an $n$-slice algebra of infinite type.
\end{proof}

Since an $n$-representation-infinite algebra is extremely Fano, so a quasi $n$-Fano algebra with Koszul $(n+1)$-preprojective algebra is extremely Fano.

\section{Algebras and  triangulated categories associated to an $n$-slice algebra\label{eqtc}}

Start with an $n$-slice algebra $\GG$ with bound quiver $Q^{\perp}$ and its quadratic dual, the $n$-properly-graded algebra $\LL = \GG^{!,op}$ with bound quiver $Q$.
In Section  \ref{pre}, we have constructed the returning arrow quivers $\tQ^{\perp}$ and $\tQ$, the $(n+1)$-preprojective algebra $\tG= \Pi(\GG)$ and its quadratic dual $\tL= \Pi(\GG)^{!,op}$, which is a twisted trivial extension of $ \LL$.
In the previous sections, we see that we have two series of  algebras parameterized by the graded automorphisms   related to the  returning quiver.
Though only with the automorphism $\sigma =\epsilon^{n+2}$, one gets Calabi-Yau, it seems the representation theory are independent of the automorphism.
Now we recall some other constructions of quivers and algebras related to an $n$-slice algebra.

Given a finite stable $n$-translation quiver $\tQ$ with $n$-translation $\tau$, we construct an infinite acyclic stable $n$-translation quiver $$\zZ_{\vv}\tQ=\zZ_{*,-} \tQ=(\zZ_{*,-} \tQ_0,\zZ_{*,-} \tQ_1, \rho_{\zZ_{*,-} \tQ}) $$ as follows (denoted by $\olQ$ and called separated directed quiver in \cite{g12} ), with the vertex set \eqqcn{vertexzv}{\zZ_{*,-} \tQ_0=\{(i,n) | i \in Q_0, n \in \mathbb Z\},} and the arrow set \eqqcn{arrowzv}{\zZ_{*,-} \tQ_1 = \{(\alpha, n):(i,n) \to (j,n+1)|\alpha: i \to j \in Q_1, n \in \mathbb Z \}.}
If $p = \alpha_s \cdots \alpha_1 $ is a path in $\tQ$, define $p[m] = (\alpha_s,m+s-1) \cdots (\alpha_1,m)$, and for an element $z=\sum_{t} a_tp_t$ with each $p_t$ a path in $\tQ$, $a_t \in k$, define $z[m]=\sum_{t} a_tp_t[m]$ for each $m \in \mathbb Z$.
Define relations
\eqqcn{relationzv}{\rho_{\zZ_{*,-} \tQ} =\{\zeta[m]| \zeta \in \trho, m \in \mathbb Z \},}
here $\zeta[m] = \sum_{t} a_t p_t[m]$ for each $\zeta = \sum_{t} a_t p_t \in \trho $.
By \cite{g12}, it is the bound quiver of the smash product $\wht{\LL^N} = \tL \#_{*,-} k\zZ^*$ with respect to the first grading of $\tL$.

The quiver $\zZ_{\vv} \tQ$ is a locally finite bound quiver if $Q$ is so.
By setting $\tau(i,t)=(\tau i,t-n-1)$, $\zZ_{*,-} \tQ$ becomes a stable $n$-translation quiver \cite{g12}.
Clearly it is an acyclic infinite $n$-translation quiver.

A quiver $Q$ is called {\em  nicely-graded}  if there is a map $d$ from $Q_0$ to $\zZ$ such that $d(j) = d(i) + 1$ for any arrow $\xa:i\to j$.
Clearly, a nicely graded quiver is acyclic.
The following properties of $\zZ_{\vv} \tQ$ are in Proposition 2.1 of \cite{gxl21}.
\begin{pro}\label{zquivers}
\begin{enumerate}
\item Let $d$ be the greatest common divisor of the length of cycles in $\tQ$, then $\zZ_{\vv} \tQ$ has $d$ connected components.

\item All the connected components of $\zZ_{\vv} \tQ$ are isomorphic.

\item Each connected component of $\zZ_{\vv} \tQ$ is nicely graded quiver.
\end{enumerate}
\end{pro}

Now assume that $Q$ is an $n$-properly-graded quiver, let $\tQ$ be its returning arrow quiver.
Conventionally, we assume that $\tQ$ is quadratic. 
Now we have another infinite acyclic stable $n$-translation quiver $$\zzs{n-1} Q =\zZ_{-,*}\tQ =(\zZ_{-,*} \tQ_0,\zZ_{-,*} \tQ_1, \rho_{\zZ_{-,*} \tQ}) $$ related to $\tQ$ with respect to the returning arrow grading, as the bound quiver  with vertex set	
$$(\zZ_{-.*} \tQ)_0 =\{(u , t)| u\in \tQ_0, t \in \zZ\},$$ 
	the arrow set
	\eqqcn{arrowzq}{\arr{rl}{(\zZ_{-,*} \tQ)_1  = & \zZ \times Q_1 \cup \zZ \times Q_{1,\caM} \\ =& \{(\alpha,t): (i,t)\longrightarrow (i',t) | \alpha:i\longrightarrow i' \in Q_1, t \in \zZ\} \\& \quad \cup \{(\beta_p , t): (i', t) \longrightarrow (i, t+1) | p\in \caM, s(p)=i,t(p)=i'  \} }}
	and the relation set  $$\rho_{\zZ_{*,-}} \tQ =\zZ \rho \cup \zZ Q_{2,\caM} \cup \zZ \rho_{\caM},$$ where
	$$\zZ \rho =  \{\sum_{s} a_s (\xa_s,t)(\xa'_s,t) |\sum_{s} a_s \xa_s \xa'_s \in \rho, t\in \zZ\},$$
	$$\zZ Q_{2,\caM} =  \{(\beta_{p'},t+1)(\beta_p ,t)| \beta_{p'} \beta_{p}\in Q_{1,\caM}, t\in \zZ\}$$ and
	$$\arr{ll}{\zZ \rho_{\caM} = &\{ \sum_{s'} a_{s'}  (\beta_{p'_{s'}}, t+1)  (\xa'_{s'}, t) +  \sum_{s} b_s (\xa_s,t+1)  (\beta_{p_s} ,t)\\& \quad\mid  \sum_{s'} a_{s'} \beta_{p'_{s'}} \xa'_{s'} + \sum_{s} b_s \xa_s \beta_{p_s} \in \rho_{\caM} , t\in \zZ\}.}$$
The quiver $\zzs{n-1} Q$ is a  locally finite infinite quiver if $Q$ is locally finite.
It is a quiver with infinite copies of $Q$ with successive two neighbours connected by the returning arrows.

With the $n$-translation $\tau$ defined by sending each vertex $(i,t)$ to $(\tau i,t-1)$, the quiver  $\zzs{n-1}Q$ becomes a stable $n$-translation quiver.

If $\tL$ is the algebra defined by the returning arrow  quiver $\tQ$ with returning arrow grading as its second  grading, then the bound quiver of smash product $\wht{\LL} = \tL \#_{-,*} k\zZ^*$ with respect to the second grading is exactly $\zZ_{-,*} \tQ =\zzs{n-1}Q$ \cite{g16}.

Let $Q$ be an $n$-nicely-graded quiver and let $\tQ$ be its returning arrow quiver.
Then $\zzs{n-1} Q$ is isomorphic to a connected component of $\zZ_{\vv} \tQ$ \cite{gxl21}.
Write  $\zzs{n-1} Q^{op} = (\zzs{n-1} Q)^{op}$,  $(\zzs{n-1} Q)^{\perp} = \zzs{n-1} Q^{\perp}$ and  $\zzs{n-1} Q^{\perp,op} = (\zzs{n-1} Q)^{\perp,op}$ for the convence in this paper.

A {\em complete $\tau$-slice} in an acyclic stable $n$-translation quiver is a full convex subquiver which intersects each $\tau$-orbit exactly once \cite{g12}.
We usually take a complete $\tau$-slice as a bound subquiver.
An algebra defined by a complete $\tau$-slice is called a {\em $\tau$-slice algebra}.

An $n$-properly-graded quiver $Q$ is a $\tau$-slice in $\zzs{n-1} Q$, so an $n$-slice is a quadratic dual of some $\tau$-slice.
So we have the following result justifying the name of $n$-slice algebra.
\begin{pro}\label{slicedual}
An $n$-slice algebra is the quadratic dual of a $\tau$-slice algebra.
\end{pro}

Consider an returning arrow quiver $\tQ$, it is known that the full bound subquiver $\zZ_{\vv} \tQ[0,n]$ of $\zZ_{*,-} \tQ$ with vertex set $\{(i,t)|i\in\tQ_0, 0\le t \le n\}$ is a $\tau$-slice in $\zZ_{\vv} \tQ$.
Write $Q^N = \zZ_{\vv} \tQ[0,n]$, and let $\LL^N$ be the algebra defined by $Q^N$, then $\tL\#_{*,-} k\zZ^*$ is the repetitive algebra of $\LL^N$, by Theorem 5.12 of \cite{g12}, and we write $\what{\LL^N} =\tL\#_{*,-} k\zZ^*$.
$\LL^N$ is the Beilinson-Green algebra of $\tL$ defined in \cite{c09}(called Beilinson algebra there).

Write $\GG^N = (\LL^N)^{!,op} $ for the quadratic dual of $\LL^N $, and write $\what{\GG^N}$ for the repetitive algebra of $\GG^N$.
Write $\what{\LL}= \tL\#_{-,*} k\zZ^*$,  similar to the argument in \cite{g12}, we also have $\what{\LL}$ is the repetitive algebra of $\LL$.

The quiver $\zzs{n-1} Q$ are important in studying the higher preprojective and preinjective components (see \cite{gll19b}).
\medskip

The representation theory of the $n$-slice algebra and of the algebras constructed above are also related.

The algebras $\GG$ and $\LL$ are Koszul dual by Section 2 of \cite{bgs96}, and  there is an equivalence between their  derived categories of finite generated graded modules as triangulated categories.
This equivalence extends to certain triangulated categories related to the  algebras constructed.

For an AS-regular algebra $\tG$, denote by $\mathrm{qgr} \tG$ the non-commutative projective scheme of $\tG$ (see \cite{az94}).
The following theorem is the generalization of a theorem stated and proved in \cite{gxl21} for the triangulated categories related to $n$-slice algebras obtained from a McKay quiver.

\begin{thm}\label{mckequi}
Let $\GG$ be an $n$-slice algebra, and let $\LL$ be its quadratic dual.

Then the following categories are equivalent as triangulated categories:
\begin{enumerate}
\item \label{ggn}  the bounded derived category $\mathcal D^b(\GG^N) $ of the finitely generated $\GG^N$-modules;
\item \label{lln}  the bounded derived category $\mathcal D^b(\LL^N) $ of the finitely generated $\LL^N$-modules;
\item \label{grtl1} the stable category $\underline{\mathrm{grmod}_{*,-}} \tL$ of finitely generated graded $\tL$-modules with respect to the first grading;
\item \label{grtdl1} the stable category $\underline{\mathrm{grmod}_{*,-}} \Delta(\LL^N)$ of finitely generated graded $\Delta(\LL^N)$-modules with respect to the first grading;
\item \label{grdg1}  the stable category $\underline{\mathrm{grmod}_{*,-}} \Delta(\GG^N)$ of finitely generated graded $\Delta(\GG^N)$-modules with respect to the first grading;
\item \label{mhtl} the stable category $\underline{\mathrm{mod}}\, \widehat{\LL^N}  $ of finitely generated $\widehat{\LL^N} $-modules;
\item \label{mhtg} the stable category $\underline{\mathrm{mod}}\, \widehat{\GG^N}  $ of finitely generated $\widehat{\GG^N} $-modules.
\end{enumerate}
If $\GG$ is of infinite type, then  they are also equivalent to the following triangulated category.
\begin{enumerate}
\item[(8)]
    the bounded derived category $\mathcal D^b( \mathrm{qgr} \tG) $ of the non-commutative projective scheme of $\tG$;
\end{enumerate}

\end{thm}
\begin{proof}
By Propositions 2.6 and 6,5 of \cite{g20}, we have $\GG^N \simeq \LL^{N,!,op}$ is Koszul.
By Proposition 2.5,  and Corollary 2.4 of \cite{g20}, we have that $\caD^b(\mmod \GG^N) $ and $ \caD^b(\mmod \LL^N)$ are equivalent to the orbit categories of $\caD^b(\grm \GG^N)$ and to the orbit categories of $\caD^b(\grm \LL^N)$, respectively, using the proof Theorem 2.12.1 of \cite{bgs96} (see the arguments before Theorem 6.7 of \cite{g20}).
So by Theorem 2.12.6 of \cite{bgs96}, $\caDb(\mmod \GG^N)$ and $\caDb(\mmod \LL^{N})$ are equivalent as triangulated categories.
This proves the equivalence of (\ref{ggn}) and (\ref{lln}).

Since $\LL^N$ is the Beilinson-Green algebra of $\tL$, we have that ${\mathrm{grmod}_{*,-}} \tL$ and ${\mathrm{grmod}_{*,-}} \Delta (\LL^N)$ are equivalent by Theorem 1.1 of \cite{c09}.
So their stable categories,  $\underline{\mathrm{grmod}_{*,-}} \tL$ and $\underline{\mathrm{grmod}_{*,-}} \Delta (\LL^N)$, are equivalent as triangulated categories.
This gives the equivalence of (\ref{grtl1}) and (\ref{grtdl1}).

By Corollary 1.2 of \cite{c09}, we have that $\underline{\mathrm{grmod}_{*,-}} \tL$ and $\caDb(\mmod \LL^N)$ are equivalent as triangulated categories.
This proves the equivalence of (\ref{grtl1}) and (\ref{lln}).

Combine  the above two  equivalences, we get the equivalence of (\ref{grtdl1}) and (\ref{lln}).
The equivalence of (\ref{grdg1}) and (\ref{ggn}) follows  similarly.

The equivalences of (\ref{lln}) and (\ref{mhtl}) and of (\ref{ggn}) and (\ref{mhtg}) are  direct consequence of Theorem II.4.9 of \cite{h88}.

If $\GG$ is of infinite type, then $\tG$ is an AS-regular algebra.
By Theorem 4.14 of \cite{mm11}, we have that $\mathcal D^b( \mathrm{qgr} \tG) $ is equivalent to $\mathcal D^b(\GG^N) $ as triangulated categories.
This is the equivalence of (\ref{ggn}) and (8) for $n$-slice algebra of infinite type.
\end{proof}

We remark that the equivalence of (\ref{ggn}) and (8) can be regarded as a generalization of Beilinson correspondence and the equivalence of (\ref{grtl1}) and (8) can be regarded as a generalization of Berstein-Gelfand-Gelfand correspondence \cite{b78,bgg78}.
So we have the following picture for the equivalences of the triangulated categories in Theorem \ref{mckequi} for infinite types:

\tiny\eqqc{depictingcat}{\xymatrix@C=2.0cm@R1cm{ \txt{$\caDb(\LL^N$)} \ar@{<->}[rr]^-{\txt{} }\ar@{<->}@[black][dr]|-{\txt{}}&& \txt{ $\caDb(\GG^N) $} \ar@{<->}@[black][dr]|-{\txt{ Beilinson equivalence }} \\
&\txt{ \color{black}{$\underline{\mathrm{grmod}} \tL $}}\ar@{<->}@[black][dl]|-{\txt{}} \ar@[black]@{<->}[rr]|-{\txt{\color{black}{BGG equivalence}}} &&\txt{\color{black}{$\mathcal D^b( \mathrm{qgr} \tG) $} }\ar@{<->}@[black][dl]\\
\txt{ \color{black}{$\underline{\mmod} \widehat{\LL^N} $}} \ar@{<->}@[black][uu] \ar@[black]@{<->}[rr]^-{\txt{\color{black}{}}} &&\ar@{<->}@[black][uu]\txt{\color{black}{$\underline{\mmod} \widehat{\GG^N}$} }}.
}\normalsize

Similar to the above, we also have the following equivalent triangulated categories, using the second grading of $\tG$.
Now we have the algebras $\widehat{\LL}  = \tL \#_{-,*} \zZ^*$ with the bound quiver $\zzs{n-1}Q$ and $\widehat{\GG}  = \tG \#_{-,*} \zZ^*$ with the bound quiver $\zzs{n-1}Q^{\perp}$.

\begin{thm}\label{tricateq1}
Let $\GG$ be a nicely-graded $n$-slice algebra, and let $\LL$ be its quadratic dual.
Then the following triangulated categories are  equivalent
\begin{enumerate}
\item[(9)]  the bounded derived category $\mathcal D^b(\GG) $ of the finitely generated $\GG$-modules;
\item[(10)]  the bounded derived category $\mathcal D^b(\LL) $ of the finitely generated $\LL$-modules;
\item[(11)]  the stable category $\underline{\mathrm{grmod}_{-,*}} \tL$ of finitely generated graded $\tL$-modules with respect to the second grading;
\item[(12)]  the stable category $\underline{\mathrm{grmod}_{-,*}} \Delta(\GG)$ of finitely generated graded $\Delta(\GG)$-modules with respect to the second grading;
\item[(13)] the stable category $\underline{\mathrm{mod}} \widehat{\LL}  $ of finitely generated $\widehat{\LL} $-modules;
\item[(14)] the stable category $\underline{\mathrm{mod}} \widehat{\GG}  $ of finitely generated $\widehat{\GG} $-modules.
\end{enumerate}
\end{thm}

By Theorems \ref{mckequi} and \ref{tricateq1}, we can approach the representation theory of a $n$-slice algebra from different point of views, if it is of infinite type, such approaches also have a non-commutative interpretation as in (8).
In \cite{hio14}, the representation theory of $n$-representation-infinite algebra is studied via the category  $\mathcal D^b( \mathrm{qgr} \tG) $.

In general, $\GG$ and $\GG^N$ are not well related, see the example 4.7 in \cite{ghl22b}.
The algebra $\GG^N$ is always nicely-graded.
When $\GG$ is nicely-graded, $\GG$ is obtained by a sequence of $n$-APR tilts from any indecomposable summand of $\GG^N$.
So it is natural to study nicely-graded $n$-slice algebras when dealing with the representation theory.

\section{Pairs of higher slice algebras of finite type\label{finite}}

Assume that $\GG$ is an $n$-slice algebra of finite type with  Coxeter index $q+1$, it is an  $n$-representation-finite algebra  with $(q+1, n+1)$-Koszul $(n+1)$-preprojective algebra $\tG$.
Let $\LL$ be the  quadratic dual of $\GG$ and let $\tL = \dtsL$ be the twisted trivial extension of $\LL$.
Let $Q=  (Q_0,Q_1,\rho)$ be the bound quiver of $\LL$, $Q^{\perp}= (Q_0,Q_1,\rho^{\perp})$ be the bound quiver of $\GG$, then $Q$ is an $n$-properly-graded quiver and $Q^{\perp} $ is an $n$-slice.
Now we have the returning arrow quiver $\tQ$ for $\tL$, its quadratic dual quiver $\tQ^{\perp}$ for $\tG$, the infinite quivers  $\zZ_{\vv}\tQ$ for the smash product $\widehat{\LL^N}=\tL\#_{*,-} \zZ^*$ and $\zzs{n-1} Q$ for the smash product $\widehat{\LL} = \tL\#_{-,*} \zZ^*$.

We now show that higher slice algebras of finite type appear in pairs.

\begin{thm}\label{fin}
Assume that $\GG$ is an acyclic $n$-slice algebra of finite type with Coxeter index $q+1$.
Then there is a $q$-slice algebra $\GG'$ with Coxeter index $n+1$ such that their repetitive algebras are quadratic dual.
\end{thm}
The $q$-slice algebra $\GG'$ is called a {\em companion of $\GG$}.
\begin{proof}
Let $\GG$ be an acyclic $n$-slice algebra with bound quiver $Q^{\perp}$, and let $\LL$ be the quadratic dual of $\GG$, with bound quiver $Q$.
Then $\tL$ and  $\wht{\LL}$ are both stable $n$-translation algebras with bound quiver the returning arrow quiver $\tQ$ and the covering quiver $\zzs{n-1} Q$, and both are  $(n+1,q+1)$-Koszul.
Now consider their respectively quadratic duals $\tG= \tL^{!,op}$ and  $\wht{\GG} = \wht{\LL}^{!,op}$, their bound quivers are respectively $\tQ^{\perp}$ and $\zzs{n-1} Q^{\perp}$.

By Theorem 6.6 of \cite{g20}, $\tL^{!,op}$ is the $(n+1)$-preprojective algebra of an $n$-representation-finite algebra $\GG$.
So by corollary 3.4 of \cite{io13}, $\tL^{!,op}$ is self-injective, and by Proposition 3.4 of \cite{bbk02}, it is $(q+1,n+1)$-Koszul.
So by Theorem 5.3 of \cite{g16}, $\wht{\LL}^{!,op}$ is also $(q+1,n+1)$-Koszul  self-injective.
This implies that $\wht{\LL}^{!,op}$ is a $q$-translation algebra and $\zzs{n-1}Q^{\perp}$ is a $q$-translation quiver with $q$-translation $\tau_{\perp}$.

Now let $Q'$ be a complete $\tau_{\perp}$-slice of $\zzs{n-1}Q^{\perp}$, then it is a $q$-properly-graded quiver by Lemma 6.1 of \cite{g20} and we have that $\zzs{n-1}Q^{\perp} \simeq \zzs{q-1} Q'$ by Proposition 3.5 of \cite{gll19b}.
Let $\LL'$ be the algebra defined by $Q'$, it is a $q$-properly-graded algebra.
Similar to the proof of Theorem 6.12 of \cite{g12}, $\zzs{n-1}Q^{\perp}$ is the bound quiver of its repetitive algebra $\wht{\LL'}$, which is $(q+1,n+1)$-Koszul, thus $\GG'={\LL'}^{!,op}$ is an $q$-slice algebra of finite type with Coxeter index $n+1$.

This proves the theorem.
\end{proof}

The algebra $\GG'$ is not unique in general, but there are only finite many of them up to isomorphism.
In fact, let $Q'$ be a complete $\tau_{\perp}$-slice of $\zzs{n-1} Q^{\perp}$, then any complete $\tau_{\perp}$-slice $Q''$ of $\zzs{n-1} Q^{\perp}$ has the same number, say $r$, of vertices as $Q'$.
Since $\tau_{\perp}$ is an automorphism of $\zzs{n-1} Q^{\perp}$, for any complete $\tau_{\perp}$-slice  $Q''$  of $\zzs{n-1} Q^{\perp}$,  by  shifting with $\tau_{\perp}^t$ if necessary, has a common  vertex with $Q'$. 
Thus for any vertex $i$ of $Q''$, there is either a path of length no longer than $2r$ from some vertex $j$ in $Q'$ to $i$, or a path of length no longer than $2r$ from $i$ to some vertex $j$ in $Q'$.
So the vertex set of the non-isomorphic  complete $\tau_{\perp}$-slices of $\zzs{n-1} Q^{\perp}$ can be chosen from  the set $V$ of  vertices of the paths of length $2r$ starting or ending at a vertex  in $Q'$, and $V$ is a finite set.
This tells us that there are only finite many  non-isomorphic  complete $\tau_{\perp}$-slices of $\zzs{n-1} Q$.

As a corollary of Theorem \ref{fin}, we have the following result on their bound quivers from its proof.
\begin{cor}\label{zzsnq}
Let  $Q^{\perp}$  be the bound quiver of an $n$-slice algebra $\GG$ of finite type with Coxeter index $q+1$ and let ${Q'}^{\perp}$ be the bound quiver of its companion $\GG'$.
Then $\zzs{n-1}Q \simeq \zzs{q-1} {Q'}^{\perp}$
\end{cor}

Now assume that  $\GG$ is nicely-graded and connected, then $\zzs{n-1}Q$ is a connected component of $\zZ_{\vv}\tQ$, while $\zzs{q-1}Q'$ is a connected component of $\zZ_{\vv}\tQ'$.
Thus the connected components of $\zZ_{\vv}\tQ$ and of $\zZ_{\vv}\tQ'$ are all the same as quivers.

We see from the proof Theorem \ref{fin} that $\tQ^{\perp}$ is a $q$-translation quiver with  $q$-translation $\tau_{\perp,\tQ}$.
Write the $n$-translation of the $n$-translation quiver $\zZ_{\vv}\tQ$ as $\tau$ and  the $q$-translation of the $q$-translation quiver $\zZ_{\vv}\tQ
^{\perp}$ as $\tau_{\perp}$.
With the index of the quiver $\zZ_{\vv}\tQ$, we have that $\tau(i,s) = (\tau i,s-n-1)$ and $\tau_{\perp} (i,s ) = (\tau_{\perp,\tQ^{\perp}}
 i,s-q-1)$, by our construction.

\bigskip
Now we study higher translation quivers related to a pair of higher slice algebras of finite type.

A quadratic quiver $\olQ=(\olQ_0,\olQ_1,\olrho)$ is called a {\em double translation quiver of type $(n,q)$} if $\olQ$ is a stable $n$-translation quiver and $\olQ^{\perp}$ is a stable $q$-translation quiver.
We obviously have the following proposition, restates Corollary \ref{zzsnq}.
\begin{pro}\label{zqdbq}
Let $\GG$ be an $n$-slice algebra of finite type with Coxeter index $q+1$ with bound quiver $Q^{\perp}$.
Then $\zzs{n-1}Q$ is a double translation quiver of type $(n,q)$.
\end{pro}

If $\olQ$ is a double translation quiver of type $(n,q)$, the algebra $\olL$ defined by the bound quiver $\olQ$ is an $n$-translation algebra and  the algebra $\olL^{\perp}$ defined by the bound quiver $\olQ^{\perp}$ is a stable $q$-translation algebra.
So the maximal bound paths of $\olQ$ define an $n$-translation $\tau$ on $\olQ$ and the maximal bound paths of $\olQ^{\perp}$ define an $q$-translation $\tau_{\perp}$.
Both defines permutations on the vertex set of $\olQ$, and each of them induces an automorphism on the quiver $\olQ$.
We obviously have the following commutative relation for them.

\begin{lem}\label{commtau}
$$\tau \tau_{\perp} = \tau_{\perp}\tau.$$
\end{lem}

Let $Q$ be a complete $\tau$-slice in $\olQ$, then $\olQ = \zzs{n-1} Q$.
So  all the double translation quivers of type $(n,q)$ are of this form.
As a corollary, we get the following result.

\begin{lem}\label{sliceundertau}
If $Q$ is a complete $\tau$-slice of $\olQ = \zzs{n-1} Q$, then $\tau_{\perp}^{t} Q$ is a complete $\tau$-slice for all the integer $t$.

If $Q'$ is a complete $\tau_{\perp}$-slice of $ \zzs{n-1} Q^{\perp}$, then $\tau^{t} Q'$ is a complete $\tau_{\perp}$-slice for all the integer $t$.

As a quiver, we have $\zzs{n-1} Q = \zzs{q-1} Q'$.
\end{lem}

A full convex subquiver $D$ of a double translation quiver $\olQ$ of type $(n,q)$ is call a {\em double slice} if it is maximal with the property that  for any vertex $\bar{i}$ in $D$, at most one of  $\tau^{-1} \bar{i}$ and $\tau_{\perp} \bar{i}$ is in $D$.

 We remark that if $\bar{i}$, $\tau^{-1} \bar{i} \in D$ and there is a bound path of length $q+1$ from $\tau^{-1} \bar{i}$ to $j$ in $\olQ^{\perp}$, then $j \notin D $. If not we have  $\tau_{\perp} \tau j=\bar{i}$ since $\tau_{\perp} j=\tau^{-1} \bar{i}$ and thus $\tau j \in D$, which contradicts the definition of double slice.

We remark that a double slice is an maximal  $\tau$-mature bound subquiver, in the term of $\tau$-mature defined  in \cite{gll19b}.

\medskip

If $\GG$ is an $n$-slice algebra of Coxeter index $q+1$ with bound quiver $Q^{\perp}$, then $\zzv \tQ$ is a double translation quiver of type $(n,q)$, and $\zzv\tQ[h, h+n+q+1]$ is a double slice for any $h$.

\medskip

Now let $\olQ$ be nicely-graded stable double translation quiver of type $(n,q)$, with $n$-translation $\tau$ and $q$-translation $\tau_{\perp}$.
If $S$ is a complete $\tau$-slice in $\olQ$, let $D(S+)$ be the full subquiver formed by $S$ and all the $\tau_{\perp}$-hammocks $H^{\bai}$ starting at each vertex $\bai$  of $S$, let $D(-S)$ be the full subquiver formed by $S$ and all the $\tau_{\perp}$-hammocks $H_{\bai}$ ending at each vertex $\bai$ of $S$.

In $D(S+)$, we have two complete $\tau$-slices, $S$ and $\tau_{\perp}^{-1} S$, and in  $D(-S)$, we have two $\tau$-slices, $S$ and  $\tau_{\perp} S$.
They are the union of $\tau_{\perp}$-hammocks connecting the vertices these two $\tau$-slices.
Note that each $\tau$-hammock is convex, so they are also convex.
\begin{pro}\label{dbslicecst}
$D(S+)$ and $D(-S)$ are convex.
\end{pro}

We clearly have that  $D(S+) =D(-\tau_{\perp}^{-1}S)$.

Write $-S_{\perp}$ for the full subquiver of $D(S+)$ with vertex set  $D(S+)_0\setminus S_0$, and write $S_{\perp}+$ for the full subquiver of $D(-S)$ with vertex set  $D(-S)_0\setminus S_0$.

\begin{pro}\label{dbslicetwo}
The full subquiver  $D(S+)$ is formed by a complete $\tau$-slice $S$, and a $\tau_{\perp}$-slice $-S_{\perp}$, connected with arrows from  non-source vertices  of  $S$ to  non-sink vertices $-S_{\perp}$.

The full subquiver $D(-S)$ is formed by a complete $\tau$-slice $S$, and a complete $\tau_{\perp}$-slice $S_{\perp}+$, connected with arrows from  non-source vertices  of  $S_{\perp}+$ to  non-sink vertices  $S$.
\end{pro}
\begin{proof}
We prove the first assertion, the second follows dually.

We prove that $-S_{\perp}$ forms a $\tau_{\perp}$-slice in $\zzs{n-1} Q^{\perp}$.
If there is a vertex $\baj$ such that the $\tau_{\perp}$-orbit of  $\baj$ has no intersection with $-S_{\perp}$, we may choose $\baj$ such that there is an arrow $\xa:\baj\to \bai$ such that $\bai$ is in $-S_{\perp}$.
But this force $\baj$ in $S$, so $\tau_{\perp}^{-1}\baj$ is in $-S_{\perp}$, this is a contradiction.
So $-S_{\perp}$ intersect each $\tau_{\perp}$-orbit by the construction, $-S_{\perp}$ intersects each $\tau_{\perp}$-orbit at most once.

Now  $-S_{\perp}$ is obtained as a union of $\tau_{\perp}$-hammocks with their sources removed, it is convex since each of these source removed $\tau_{\perp}$-hammocks are convex, and all the removed sources form a convex set which preceding  $-S_{\perp}$.

This proves that $-S_{\perp}$ forms a $\tau_{\perp}$-slice in $\zzs{n-1} Q^{\perp}$.
\end{proof}

\begin{pro}\label{dbslicecst} $D(S+)$ and $D(-S)$ are double slices.
\end{pro}
\begin{proof} We prove that $D(S+)$ is a double slice.
The other assertion follows from duality.

For each vertex $\bai$ in $D(S+)$, if $\bai $ and $\tau_{\perp}^{-1} \bai$ are in $D(S+)$, then by the construction $\bai$ is in $S$, so $\tau \bai$ is not $S$ and hence not in $D(S+)$.

Now  $\tau_{\perp}^{-1} S$  also forms a complete $\tau$-slice isomorphic to $S$.
If $\baj$ is a vertex not in $D(S+)$ such that there is an arrow from $\baj$ to a vertex $\bai'$ in $S$, then we have $\tau^{-1} \baj $ is in $S$ and  $\tau_{\perp}^{-1}\tau^{-1} \baj $ in $D(S+)$.
If $\baj$ is a vertex not in $D(S+)$ such that there is an arrow from $\bai'$ in $\tau_{\perp}^{-1}S$ to $\baj$, then we have $\tau \baj $ in $\tau_{\perp}^{-1}S$ and  $\tau_{\perp}\tau \baj $ in $S$.

This shows that $D(S+)$ is maximal, and hence a double slice.
\end{proof}

Note that the vertex set of $\tau (-S_{\perp})$ is also in $D(S)$ and it is a complete $\tau_{\perp}$-slice in $\zzs{n-1} Q^{\perp}$.
Regard as a double translation quiver of type $(q,n)$, we also construct the double slice $D(\tau (-S_{\perp}))$ and $D(\tau^{-1} (S_{\perp}+)) $ in $\zzs{n-1} Q^{\perp}$, and obviously have the following.
\begin{pro}\label{dbslicedual}
 $D(\tau (-S_{\perp})) = D(S+)^{\perp}$, and  $D(\tau^{-1} (S_{\perp}+)) = D(-S)^{\perp}$.
\end{pro}

Let $Q^{\perp}$ be the bound quiver of the $n$-slice algebra $R$, then $\zzs{n-1} Q$ is a double translation quiver of type $(n,q)$, and $S=Q$ is a $\tau$-slice.
The algebra $R^{-,\perp}$ defined by $\tau(-S_{\perp})$ and $R^{+,\perp}$ defined by $\tau^{-1}(S_{\perp}+)$ are companions of $R$.
Called the {\em right} and {\em left} companion of $R$, respectively.
We usually take the right companion of $R$ as default one and write it as $\GG^c$.

Recall that the complete $\tau$-slices are related by the $\tau$-mutations defined as follow.
If $i$ is a sink of a complete $\tau$-slice $S$, we define the {\em $\tau$-mutation} $ s_i^-S$ of $S$ at $i$ as the full bound subquiver of $\olQ $ obtained by replacing the vertex $i$ by its $n$-translation $\tau i$.
If $i$ is a source of a complete $\tau$-slices $S$, define the {\em $\tau$-mutation} $s_i^+S$ of $S$ at $i$ as the full bound subquiver in $\olQ $ obtained by replacing the vertex $i$ by its inverse $n$-translation $\tau^{-1} i$.

Double slices are also related by mutations.
{\em The  mutation $s^i D$ of double slice $D$} with respect to a source $i$ is obtained by removing the vertex $i$ and add the vertex $\tau_{\perp}^{-1}\tau^{-1} i$ and the arrows to $\tau_{\perp}^{-1}\tau^{-1} i$ in $\olQ$; {\em the  mutation $s_i D$ of double slice $D$} with respect to a sink $i$ is obtained by removing the vertex $i$ and add the vertex $\tau\tau_{\perp} i$ and the arrows from $\tau\tau_{\perp} i$ in $\olQ$.

\begin{pro}\label{dbslicemut} The $\tau$-mutation $s^{\bai}$ on $S$ induces a mutation $s^{\bai}$ on $D(S+)$;  The $\tau$-mutation $s_{\bai}$ on $S$ induces a mutation $s_{\bai}$ on $D(-S)$.
\end{pro}
\begin{proof}
We prove the first assertion, the second follows from duality.

By apply $\tau$-mutation $s^{\bai}$ on $S$ to get $s^{\bai}S$, we remove the source $\bai$ from $S$ and add $\tau^{-1}\bai$ as a sink of $s^{\bai}S$.

Since $n, q\ge 2$, if there are arrow $\bai \to \baj_s$, $s= 1,\cdots,h$ for some $h$, then $\baj_s$ is in $S$ and all the vertices of the $\tau_{\perp}$-hammocks $H^{\bai}$ except $\bai$ is contained the union of these $H^{\baj_s}$.
Similarly, for the arrows $\baj'_{s'} \to \tau^{-1} \bai$, all the vertices of the $\tau_{\perp}$-hammocks $H^{\tau^{-1}\bai}$ except $\tau^{-1}_{\perp}\tau^{-1}\bai$ are contained the union of these $H^{\baj'_{s'}}$.

This proves $s^{i} D(S+) = D(s^{i} S+)$.
\end{proof}

By Lemma 6.5 of \cite{g12} and the above Proposition, we have the following result.
\begin{pro}\label{dbslicemutdb}
Let $D$ and $D'$ be two double slice of a stable double translation quiver.
Then there is a sequences of mutations $s_1,\cdots,s_l$ such that $s_l\cdots s_1 D =D'$.
\end{pro}

Now assume that $Q$ is nicely-graded and connected, then $\zzs{n-1} Q$ and $\zzs{n-1} Q^{\perp}$ are both nicely-graded and connected.
So these quivers  are connected component of $\zZ_{\vv}\tQ$.
Using the index of  the vertices of  $\zZ_{\vv}\tQ$ for the vertices of $\zzs{n-1} Q$ and $\zzs{n-1} Q^{\perp}$.
Now for any $t$,  $\zZ_{\vv}\tQ[t,t+q+n+1]$ is a double slice of $\zZ_{\vv}\tQ$, and we can take $D^0=\zZ_{\vv}\tQ[-n,q+1]$, then any double slice can be obtained from $D^0$ by a sequences of mutations by Proposition \ref{dbslicemutdb}.

\section{The higher preprojective components\label{hprec}}

Now we study the representations of the pair of higher slice algebras of finite type, showing that they share the same quiver as the Auslander-Reiten quiver for their higher preprojective components.

For a finite dimensional algebra $\GG$ of global dimension $n$, the $n$-Auslander-Reiten translations of $\GG$ modules are introduced by Iyama (see \cite{i007,i07}), $$\ttn{-1} =\tr D\OO^{1-n} = \Ext_{\GG}^n(D -, \GG)\quad and \quad \ttn{}   = D\tr \OO^{n-1} =D\Ext_{\GG}^n(-, \GG), $$ with the convention that $\ttn{0}=\ttn{-0}=\identity$.

Modules in  \eqqcn{}{ \caM^+ (D\GG) =\add \{ \ttn{t} D\GG | t \ge 0\}  \mbox{ and }  \caM^- (\GG) =\add \{ \ttn{-t} \GG | t \ge 0\}}
are called {\em $n$-preinjective modules  of $\GG$} and {\em $n$-preprojective modules  of  $\GG$}, respectively  (see \cite{hio14}).
These modules  are natural generalization of preinjective and preprojective modules, respectively, thus these classes  are  called {\em $n$-preprojective component} and {\em $n$-preinjective component}, respectively.

\medskip

Now assume that $\GG$ is an acyclic $n$-slice algebra of finite type with bound quiver $Q^{\perp}$, then $\GG$ is $n$-representation-finite.
By Proposition 1.12 of \cite{i11}, we have that $\caM^+(D\GG)=\caM^-(\GG)$,  which we write as $\caM(\GG)$.
By Proposition 2.3 of \cite{io11}, $\caM(\GG)$ is an $n$-cluster tilting subcategory and by Theorem 3.3.1 of \cite{i007},  $\caM(\GG)$ has $n$-almost split sequences.

Assume that $\GG$ is nicely-graded and connected, so $Q$ is nicely-graded connected quiver  and each connected components of $\zzv\tQ$ is isomorphic to $\zzs{n-1}Q$, which is an $n$-translation quiver, write its $n$-translation as $\tau$.

Let $\olL$ be the algebra defined by the bound quiver  $\zzs{n-1} Q$ and let $\olG$ be the algebra defined by the bound quiver $\zzs{n-1} Q^{\perp}$.
Since $\tL$ is an $n$-translation algebra, $\olL$ is an $n$-translation algebra with Coxeter index $q+1$, by Theorem 5.3 of \cite{g16}.
So $\olG$ is $(q+1,n+1)$-Koszul.
By Corollary \ref{zzsnq}, $\zzs{n-1} Q^{\perp}$ is a $q$-translation quiver, write $\tau_{\perp}$ for its $q$-translation.
Identify $Q$ with the complete $\tau$-slice $\zzs{n-1} Q[0]$ of $\zzs{n-1}Q$ and write $i$ for the vertex $(i,0)$ in $Q$  when needed.

\medskip

Let $d$ be a function $d:\zzs{n-1}Q\to \zZ$ such that $d(j,t')=d(i,t)+1$ if there is an arrow from $(i,t)$ to $(j,t')$, we may assume that $\mathrm{max}\{d(i)|i \in Q_0\}=0$. 
Fix a vertex $(i,t)$ in $\zzs{n-1}Q$, set $$\mfd_{(i,t)}(j,t') = (t-t')(n+1)+d(i)-d(j).$$
Since $\zzs{n-1}Q$ is nicely-graded, a $\tau$-hammock $H_{(i,t)}$ \cite{g12,gll19b} in $\zzs{n-1} Q$ ending at a vertex $(i,t)$ is identified with the full subquiver of $\zzs{n-1} Q$  with vertex set 
\small
$$H_{(i,t), 0}=\left\{ (j,t') \mid (j,t') \in \zzs{n-1}Q_0, \exists \mbox{ path } 0\neq p\in \olL_{\mfd_{(i,t)}(j,t')}, s(p)=(j,t'), t(p)=(i,t) \right\}.$$ 
\normalsize
The hammock function $\mu_{(i,t)}: H_{(i,t),0} \longrightarrow \zZ$ is an integral map on the vertices defined by $$\mu_{(i,t)} (j,t')= \dmn_k e_{(i,t)} \olL_{\mfd_{(i,t)}(j,t')} e_{(j,t')}.$$

For each $(i,t) $ in $\zzs{n-1}Q$, we have a Koszul complex
\eqqc{nassinsubG}{\mathbf{M}_{(i,t)}: M_{n+1}=\olG e_{(i,t)}  \slrw{f_{n+1}} \cdots \to M_r \to \cdots \slrw{f_0} M_0= \olG e_{(i,t-1)} }
in $\add \olG$,  with $$M_r = \bigoplus_{\stackrel{(j,t')\in H_{(i,t),0}}{ \mfd_{(i,t)}(j,t')=n+1-r}} (\olG e_{(j,t')})^{\mu_{(i,t)}(j,t')} $$ for $0\le r\le n+1$.
This is the projective resolution of the simple $\olG$-module $\olG_0 e_{(i,t-1)}$(Proposition 2.1 of \cite{gll19b}, see also Proposition 7.4 of \cite{g16}).
Since $\olG$ is $(q+1,n+1)$-Koszul algebra, we have that $$\Ker f_{n+1} \simeq\olG_0 e_{\tau_{\perp}^{-1}(i,t)},$$ the simple $\olG$-module $\olG e_{\tau_{\perp}^{-1} (i,t)}$ corresponding to the vertex $\tau_{\perp}^{-1} (i,t)$. 

It is obvious that, regarding as a bound quiver, each convex full subquiver $T$ of $\zzs{n-1} Q$ is an $n$-translation quiver and it defines an $n$-translation algebra $\LL(T)$ with Coxeter number $q+1$ (so it is an $(n+1,q+1)$-algebra).
Write its quadratic dual as $\GG(T)$. 
Then $$\LL(T)= e_T\olL e_T \mbox{ and }\GG(T) = e_T\olG e_T$$ for the idempotent $e_T=\sum_{(i,t)\in T_0} e_{(i,t)}$.
Let $T$ be a convex full subquiver with $Q$  a terminal complete $\tau$-slice, then $\add \GG$ is full subcategory of $$\caG(T) = \add\GG(T) \simeq \add \{\olG e_{(i,t)}\mid (i,t)\in T_0\}.$$
By Lemma 7.1 of \cite{g16}, we have the following Lemma. 
\begin{lem}\label{nassT}
If $H_{(i,t)}$ is in $T$, then $e_{T} \bM{(i,t)}$ is an $n$-almost split sequence in $\caG(T)$ if and only if $\Ker f_{n+1}|_T =0$.
\end{lem}

Let $D = D(-Q)$ be the double slice with $Q$ a terminal complete $\tau$-slice, then $D^{op}$ is a double slice in $\zzs{n-1} Q^{op}$.
So we have that $Q^{op}$ is embedded into $D^{op}$ as an initial complete $\tau$-slice.

Let  $\caG(D)$ be the bound path category  defined by the bound quiver $D$ with relation set the restriction of $\rho^{\perp}_{\zzs{n-1}Q}$.
We have that $$\caG(D) \simeq \add\{\olG e_{(j,r)} | (j,r) \in D_0 \} \simeq  \add\{e_D \olG e_{(j,r)} | (j,r) \in D_0 \}$$ by Proposition 2.6 of \cite{gll19b}.
Since $D$ is a double slice, if $(i,-t)$ and $(i,-t-1)$, are in $D$, then $\tau^{-1}_{\perp}(i,-t)$ is not in $D$ and $\Ker f_{n+1}\mid_D =0$ in  $e_D\bM{(i,-t)}$, by Lemma \ref{nassT} and Theorem 7.5 of \cite{g16}, we have the following Lemma.

\begin{lem}\label{nassGD} $\caG(D)$ has $n$-almost split sequence.

If both $(i,-t)$ and $(i,-t-1)$ are in $D$, then $e_D\bM{(i,-t)}$ \eqref{nassinsubG} is an $n$-almost split sequence in  $\caG(D)$. 
\end{lem}

Let $e=e_Q$, then  $\GG = e\olG e $.
Write $Q(\caM(\GG))$ for the full subquiver of $\zZ|_{n-1} Q^{op}$ with vertex set  \eqqcn{}{Q_0(\caM(\GG)) = \{(j,-r) \in \zZ|_{n-1} Q^{op}_0 \quad | \ttn{-r} \GG e_j  \neq 0\}} for $j \in Q^{op}_0$.
\begin{lem}\label{QcaM}
The bound subquiver $Q(\caP(\GG)) $ of $\zzs{n-1}Q^{op}$ is  convex.
\end{lem}
\begin{proof}
Write $Q(\caP(\GG))=Q $ for the full subquiver of $Q(\caM(\GG)) $ with vertex set $$Q_0(\caP(\GG)) = \{(j,0) \in \zZ|_{n-1} Q^{op}_0 \quad | j \in Q^{op}_0\} $$ and $Q(\caI(\GG)) $ for the full subquiver of $Q(\caM(\GG)) $ with vertex set $$Q_0(\caI(\GG)) = \{(j,-r) \in Q_0(\caM(\GG))  \quad | \ttn{-r-1} \GG e_j  = 0  \} .$$ 
By Proposition 1.12 of \cite{i11}, $Q(\caP(\GG))$ and   $Q(\caI(\GG))$ are isomorphic, and they form complete $\tau$-slices. By definition, the vertices in $Q(\caM(\GG))$ are exactly vertices lying between $Q(\caP(\GG))$ and  $Q(\caI(\GG))$, so $Q(\caM(\GG))$ is convex.
\end{proof} 

The vertices of $\zzs{n-1}Q$ is totally ordered such that $(i,t) < (j,t')$ if $t < t'$, or $t=t' $ and $d(i) < d(j)$.
For a convex subquiver $T$, write $T(i,t)$ the subquiver of $T$ consisting of the vertives $(j,t')$ with $(j,t')> (i,t)$, and  $T[i,t]$ the subquiver of $T$ consisting of the vertives $(j,t')$ with $(j,t') \ge (i,t)$.
Let $(i_0,0)$ be the element in $Q$ which is maximal with respect to this order.

Now we show that the subquivers $Q(\caM(\GG))$ and $D^{op}$ are the same.

\begin{lem}\label{eqDqM} $D^{op} = Q(\caM(\GG)$
\end{lem}

\begin{proof}
We prove that they have the same vertex set, that is $D_0 = Q_0(\caM(\GG))$.

Note that $\caG(D)$ is an Orlov category, so $\Hom_{\caG(D)}(\olG e_{(i,-t)},\olG e_{(j,-t')} = 0$ whenever $(j, -t') > (i,-t)$. 
By  Proposition 5.3 of \cite{gll19b}, $\caM(\GG)$ is an Orlov category, with respect to the order defined above.
Set  
\small $$\caG(D)(j,-r) = \add\{\olG e_{(i,-t)}\mid (i,-t) \in D_0(j,-r)\},$$ $$\caG(D)[j,-r] = \add\{\olG e_{(i,-t)}\mid (i,-t) \in D_0[j,-r]\},$$ 
$$\caM(\GG)(j,-r) = \add\{\ttn{-t}\GG e_{i}\mid (i,-t) \in Q_0(\caM(\GG))(j,-r)\}$$\normalsize and \small$$\caM(\GG)[j,-r] = \add\{\ttn{-t}\GG e_{i}\mid  (i,-t) \in Q_0(\caM(\GG))[j,-r]\}.$$\normalsize 
Then  \small$$D^{op}(i_0,-1)=Q^{op} =Q(\caM(\GG))(i_0,-1),$$  $$\caG(D)(i_0,-1)  \simeq \caM(\GG)(i_0,-1) \simeq \add \GG,$$\normalsize  and we have for each $j\in Q_0$, $$e\olG e_{(j,0)} = \GG e_j.$$

Now assume for a vertex $(i,-r-1)$, $$\caG(D)(i,-r-1) \simeq \caM(\GG)(i,-r-1)  $$ and for each $(j,-t) \in D_{0}(i,-r-1)$, $$e\olG e_{(j,-t)} \simeq \ttn{-t}\GG e_j.$$
If $(i,-r),(i,-r-1)\in D_0$, then $e_D \bM{(i,-r)}$ is an $n$-almost split sequence in $\caG(D)$, by Lemma \ref{nassGD}.
So $e \bM{(i,-r)}$ is a sink sequence in $\caG(D)[i,-r-1]$, hence also a  sink sequence in $\caM(\GG)[i,-r-1]$. 
Since $\caM(\GG)$ is Orlov, it is a sink sequence in $\caM(\GG)$.
By Proposition 3.3 of \cite{i007},  $e \bM{(i,-r)}$ is an $n$-almost split sequence in  $\caM(\GG)$ and by Theorem 3.3.1 of \cite{i007}, $$e \olG e_{(i,-r-1)} \simeq \ttn{-1}e\olG e_{(i,-r)} \simeq  \ttn{-r-1}\GG e_{i}.$$

So $\caG(D)[i,-r-1]$ is equivalent to $\caM(\GG)[i,-r-1]$.
Inductively, we see $$D_0 \subset Q_0(\caM(\GG))$$ and $$\caG(D) \simeq \add \{\olG e_{(i,-t)}\mid (i,-t)\in D_0\}$$ is regarded as a full subcategory of $$\caM(\GG) \simeq \add\{\olG e_{(i,-t)}\mid (i,-t)\in  Q_0(\caM(\GG))\}.$$

If $(i,-r-1) \not \in D_0$, we have both $\tau^{-1}(i,-r-1)=(i,-r)$ and $\tau_{\perp}(i,-r)=(i',-t')$ are in $D$, since $D$ is double slice. 

Since $\bM{(i,-r)}$ is a sink sequence in $\add \olG$, it is a sink sequence in $\add\{\caG(D),\olG e_{(i,-r-1)}\}$, since $\add \olG$ is Orlov. 

If $(i,-r-1) \in Q_0(\caM(\GG))$ then $\ttn{-r-1} \GG e_i \neq 0$, $e\bM{(i,-r)}$ is a sink sequemce and hence an $n$-almost split sequence in $\caM(\GG)$, by  Proposition 3.3 of \cite{i007} and $$e\olG e_{(i,-r-1)} \simeq \ttn{-r-1}\GG e_{i}.$$
So $(e_D+e_{(i,-r-1)}) \bM{(i,-r)}$ is also a source sequence in $\caG(D\cup(i,-r-1))$, where $D\cup(i,-r-1)$ is the full bound subquiver of $\zzs{n-1}Q$ with vertex set $D_0\cup\{(i,-r-1)\}$.
But as the complex of $\GG(D\cup(i,-r-1))$, $\Ker f_{n+1}\mid_{\GG(D\cup(i,-r-1))} \simeq \olG_0 e_{\tau_{\perp} (i,-r)}$ which is not zero since $\tau_{\perp} (i,-r)$ is in $D$.
This contradicts Lemma \ref{nassT}.

So  $(i,-r-1) \not \in Q_0(\caM(\GG))$.
This proves $D^{op} = Q(\caM(\GG)$.
\end{proof}

 As a corollary of Lemma \ref{eqDqM}, we obtain the following theorem.

\begin{lem}\label{cateq} We have the equivalence
$\caG(D) \simeq \caM({\GG})$.
\end{lem}

{\bf Remark}. From the proof of Lemma \ref{cateq}, we see the assumption of ``$\tau$-mature'' in  Theorem 5.6 of \cite{gll19b} can be removed.

\bigskip

As an immediate consequence of Lemma \ref{cateq}, we have the following result on the Auslander-Reiten quiver of $n$-slice algebra of finite type.
\begin{thm}\label{narq} Let $\GG$ be an $n$-slice algebra of finite type with nicely-graded bound quiver $Q^{\perp}$.
Then the Auslander-Reiten quiver of its $n$-preprojective component of  $\GG$ is the opposite quiver of a double slice in $\zzs{n-1} Q$.
\end{thm}

More precise, we have
\begin{thm}\label{finARq} Let $\GG$ be an $n$-slice algebra of finite type  with nicely-graded bound quiver $Q^{\perp}$ and $q+1$ as its Coxeter index and let $\caQ$ be the Auslander-Reiten quiver of its $n$-preprojective modules.
Then
\begin{enumerate}
\item There is a $q$-slice algebra  $\GG^c$ of finite type such that $\caQ$ is the Auslander-Reiten quiver of the $q$-preprojective component of  $\GG^c$.

\item $\caQ = D(Q)^{op}$.

\item As a quiver $\caQ$ is the opposite quiver of a quiver obtained by connecting the quiver $Q^{\perp}$ of $\GG$ and the quiver ${Q^c}^{\perp}$ of $\GG^c$ by some arrows.
\end{enumerate}
\end{thm}

Note that the category of $n$-preprojective modules of a $n$-representation-finite algebra is an $n$-cluster tilted subcategory Proposition 2.6 of \cite{hio14}.
We have the following result.
\begin{thm}\label{hcluster} Let $\GG$ be an $n$-slice algebra of finite type with Coxeter index $q+1$, and let $\GG^c$ be its companion.
Then the $n$-cluster tilted subcategory of $\GG$ and the $q$-cluster tilted subcategory of $\GG^c$ are quadratic duals.
\end{thm}

As a corollaries, we have the following results for hereditary algebras, using Corollary 4.3 of \cite{bbk02}.

\begin{cor}\label{herefin} Let $\GG$ be an hereditary algebra of finite type and let $h$ be Coxeter number of its quiver.
Then there is an $(h-3)$-slice algebra $\GG^c$ of finite type such that the module category of $\GG$ is the quadratic dual of the $(h-3)$-cluster tilted subcategory of $\GG^c$-modules.
\end{cor}

Using the higher Auslander algebra introduced in \cite{i08}, one gets the following Corollary.

\begin{cor}\label{highAA} For each $n$-Auslander algebra $\caA$ of an n-slice algebra of finite type, there exist a $q$ and a $q$-Auslander algebra $\caA^c$  of a q-slice algebra of finite type such that $\caA$ and $\caA^c$ are quadratic dual one another.
\end{cor}

It is interesting to know if Theorem \ref{hcluster} holds for higher cluster tilting subcategories of finite type in general?

\bigskip

%\subsection{Example}

Higher slice algebras of finite type appear in  pairs, the following is an example.

\begin{exa}\label{E:one}{\em
The following example was presented in \cite{gx21} to illustrates the $\tau$-hammocks and $\tau$-mutations for $n$-slice algebra of finite type.
Now we consider its dual higher slice algebra.
The Auslander algebra $\GG $ of the path algebra  of type $A_3$ with linear orientation, is a $2$-representation-finite algebra, given by the quiver $Q$:
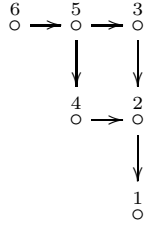
$$
\xymatrix@C=0.4cm@R0.6cm{
&& \stackrel{6}{\circ} \ar[r] &\stackrel{5}{\circ} \ar[r]\ar[d] &\stackrel{3}{\circ}\ar[d]\\
&&              &\stackrel{4}{\circ} \ar[r]       &\stackrel{2}{\circ}\ar[d] &{} \\
&&                                 &&\stackrel{1}{\circ} &{} \\
}
$$
with $\xa_i$ for the vertical arrow ending at vertex $i$, and $\xb$ for the horizontal arrow ending at vertex $i$.
The the relation set $\rho =\{\xa\xa , \xb\xb, \xa\xb -\xb\xa\}$ and $\rho^{\perp} =\{ \xa\xb +\xb\xa\}\cup \{\xa_1\xb_2,\xa_4\xb_5\}$.
The returning arrow quiver $\tQ$,
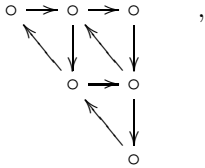
$$
\xymatrix@C=0.4cm@R0.6cm{
&& \circ \ar[r] &\circ \ar[r]\ar[d] &\circ\ar[d] &{}\\
&&              &\circ \ar[r] \ar[ul]      &\circ\ar[d] \ar[ul] &{} \\
&&                                 &&\circ \ar[ul] &{} \\
},
$$
is obtained from $Q$ by adding the arrow $\xc_{i}$  going up left to vertex $i$.
The relation set $\trho =\{\xa\xa , \xb\xb, \xc\xc, \xa\xb -\xb\xa, \xa\xc-\xc\xa, \xb\xc -\xc\xb \}$, and its dual relation set $\trho^{\perp} =\{ \xa\xb +\xb\xa, \xa\xc+\xc\xa, \xb\xc +\xc\xb\}\cup \{\xa_1\xb_2,\xa_4\xb_5, \xb_2\xc_4,\xb_3\xc_5, \xc_5\xa_2, \xc_6\xa_4\}$.
For each rhombus formed by two paths of length $2$, we have commutative relations in both relation sets.
We have zero relations for paths of length $2$ of the same type (going in the same direction in our presentation) in $\trho$, while zero relations in $\trho^{\perp}$ are the complements in the cyclic paths of length 3 of the arrows $\xa_1,\xa_2,\xb_3,\xb_5,\xc_4,\xc_6$.

With the  relation set $\trho $, $\tQ$ becomes a $2$-translation quiver with the trivial translation, and with quadratic dual relation $\trho^{\perp}$, $\tQ$ becomes  a $1$-translation quiver with translation defined by $\tau 1=3,  \tau 2=5, \tau 3= 6, \tau 4 = 2, \tau 5 = 4, \tau 6 =1$.

The $3$ connected components of $\zZ_{\vv}\tQ$ look the same, they are isomorphic to  the quiver $\zzs{1}Q$   as follows.

$$
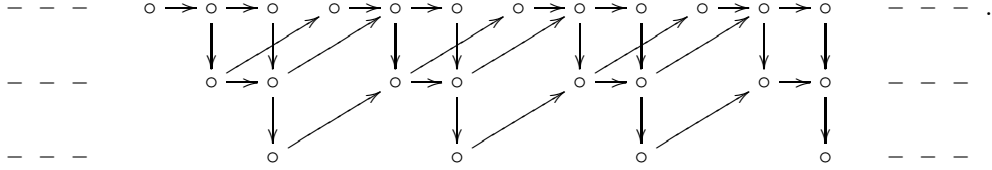
\xymatrix@C=0.4cm@R0.6cm{
&&\ar@{--}[ll]& \circ \ar[r] &\circ \ar[r]\ar[d] &\circ\ar[d] &{}  \circ \ar[r] &\circ \ar[r]\ar[d] &\circ\ar[d] &{} \circ \ar[r] &\circ \ar[r]\ar[d] &\circ\ar[d] &{} \circ \ar[r] &\circ \ar[r]\ar[d] &\circ\ar[d] &\ar@{--}[rr]&&\\
&& \ar@{--}[ll]  &           &\circ \ar[r] \ar[urr]      &\circ\ar[d]\ar[urr] &{}              &\circ \ar[r]\ar[urr]       &\circ\ar[d]\ar[urr] &{}              &\circ \ar[r] \ar[urr]      &\circ\ar[d]\ar[urr] &{}              &\circ \ar[r]      &\circ\ar[d] &\ar@{--}[rr]&&\\
&& \ar@{--}[ll] &                               &&\circ\ar[urr] &{}                                 &&\circ\ar[urr] &{}                                 &&\circ\ar[urr] &{}                                 &&\circ &\ar@{--}[rr]&&\\
}.
$$

The $\tau$-hammocks are as follows.
$$
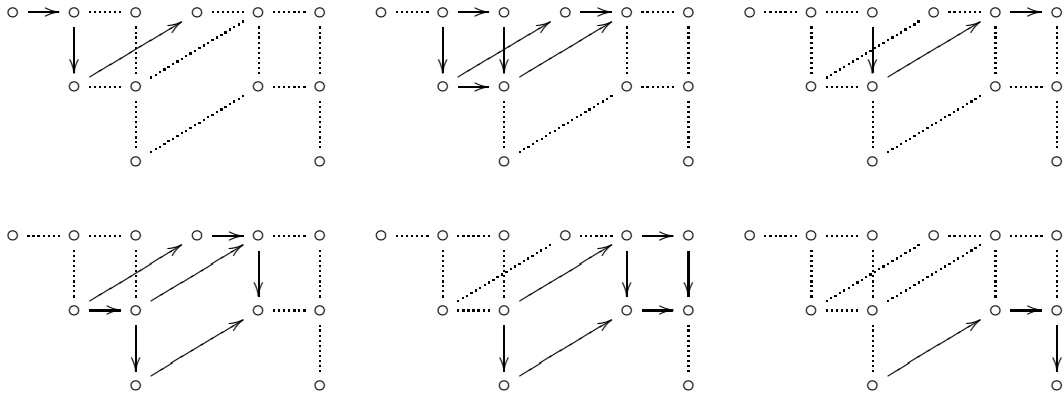
\xymatrix@C=0.4cm@R0.6cm{
\circ \ar[r] &\circ \ar@{.}[r]\ar[d] &\circ\ar@{.}[d] &{}  \circ \ar@{.}[r] &\circ \ar@{.}[r]\ar@{.}[d] &\circ\ar@{.}[d] &{}{} \circ \ar@{.}[r] &\circ \ar[r]\ar[d] &\circ\ar[d] &{} \circ \ar[r] &\circ \ar@{.}[r]\ar@{.}[d] &\circ\ar@{.}[d] &{}{} \circ \ar@{.}[r] &\circ \ar@{.}[r]\ar@{.}[d] &\circ\ar[d] &{} \circ \ar@{.}[r] &\circ \ar[r]\ar@{.}[d] &\circ\ar@{.}[d]&&\\
             &\circ \ar@{.}[r] \ar[urr]      &\circ\ar@{.}[d]\ar@{.}[urr] &{}              &\circ \ar@{.}[r]     &\circ\ar@{.}[d] &{}{}              &\circ \ar[r] \ar[urr]      &\circ\ar@{.}[d]\ar[urr] &{}              &\circ \ar@{.}[r]      &\circ\ar@{.}[d] &{}{}              &\circ \ar@{.}[r] \ar@{.}[urr]      &\circ\ar@{.}[d]\ar[urr] &{}              &\circ \ar@{.}[r]      &\circ\ar@{.}[d] &&\\
                                &&\circ\ar@{.}[urr] &{}                                 &&\circ &{}{}                                 &&\circ\ar@{.}[urr] &{}                                 &&\circ &{}{}                                 &&\circ\ar@{.}[urr] &{}                                 &&\circ&&\\%first
\circ \ar@{.}[r] &\circ \ar@{.}[r]\ar@{.}[d] &\circ\ar@{.}[d] &{}  \circ \ar[r] &\circ \ar@{.}[r]\ar[d] &\circ\ar@{.}[d]  &{}{} \circ \ar@{.}[r] &\circ \ar@{.}[r]\ar@{.}[d] &\circ\ar@{.}[d] &{} \circ \ar@{.}[r] &\circ \ar[r]\ar[d] &\circ\ar[d] &{} \circ \ar@{.}[r] &\circ \ar@{.}[r]\ar@{.}[d] &\circ\ar@{.}[d] &{} \circ \ar@{.}[r] &\circ \ar@{.}[r]\ar@{.}[d] &\circ\ar@{.}[d]&&\\
             &\circ \ar[r] \ar[urr]      &\circ\ar[d]\ar[urr] &{}              &\circ \ar@{.}[r]     &\circ\ar@{.}[d] &{}{}              &\circ \ar@{.}[r] \ar@{.}[urr]      &\circ\ar[d]\ar[urr] &{}              &\circ \ar[r]      &\circ\ar@{.}[d] &{}{}              &\circ \ar@{.}[r] \ar@{.}[urr]      &\circ\ar@{.}[d]\ar@{.}[urr] &{}              &\circ \ar[r]      &\circ\ar[d] &&\\
                                &&\circ\ar[urr] &{}                                 &&\circ &{}{}                                 &&\circ\ar[urr] &{}                                 &&\circ &{}{}                                 &&\circ\ar[urr] &{}                                 &&\circ&&\\%second
}.
$$

The homogeneous $2$-slice quivers of finite type obtained by take connected component in $\zzv \tQ^{\perp}$:

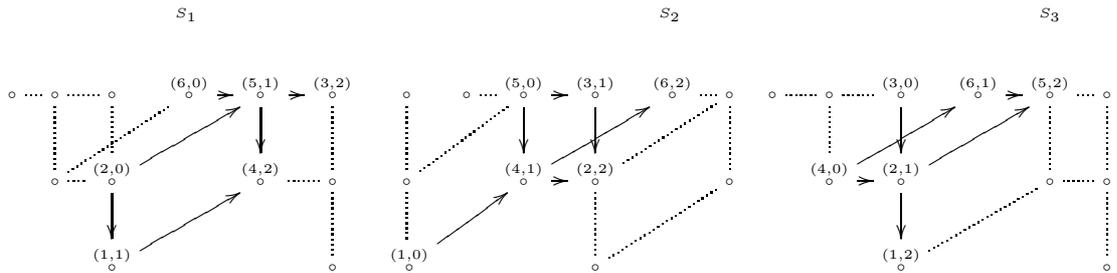
\Tiny
$
\xymatrix@C=0.2cm@R0.6cm{&&&S_1&&& &{}{} &&&S_2&& &{}{}{} &&&S_3&&&\\
 \circ \ar@{.}[r] &\circ \ar@{.}[r]\ar@{.}[d] &\circ\ar@{.}[d] &{}   \stackrel{(6,0)}{\circ} \ar[r] &\stackrel{(5,1)}{\circ} \ar[r]\ar[d] & \stackrel{(3,2)}{\circ}\ar@{.}[d] &{}{}\circ\ar@{.}[d] &{}   {\circ} \ar@{.}[r] &\stackrel{(5,0)}{\circ} \ar[r]\ar[d] & \stackrel{(3,1)}{\circ}\ar[d] &{} \stackrel{(6,2)}{\circ} \ar@{.}[r] &\circ\ar@{.}[d] &{}{}{} {\circ} \ar@{.}[r] &{\circ} \ar@{.}[r]\ar@{.}[d] & \stackrel{(3,0)}{\circ}\ar[d] &{} \stackrel{(6,1)}{\circ} \ar[r] & \stackrel{(5,2)}{\circ} \ar@{.}[r]\ar@{.}[d] &\circ\ar@{.}[d] &{}\\
 &\circ \ar@{.}[r] \ar@{.}[urr]      &\stackrel{(2,0)}{\circ}\ar[d]\ar[urr] &{}              &\stackrel{(4,2)}{\circ}  \ar@{.}[r]    &\circ\ar@{.}[d] &{} {}  {\circ}\ar@{.}[d]\ar@{.}[urr] &{}              &\stackrel{(4,1)}{\circ}\ar[r]\ar[urr]       &\stackrel{(2,2)}{\circ}\ar@{.}[d]\ar@{.}[urr] &{}              &\circ  &{}{}{}   &\stackrel{(4,0)}{\circ}\ar[r]\ar[urr]       &\stackrel{(2,1)}{\circ}\ar[d]\ar[urr] &{}              &\circ \ar@{.}[r]       &\circ\ar@{.}[d] &{}       \\
 &&\stackrel{(1,1)}{\circ}\ar[urr] &{}                                 &&\circ &{}{}\stackrel{(1,0)}{\circ}\ar[urr] &{}                                 &&\circ\ar@{.}[urr] &{}                                 & &{}{}{}&&\stackrel{(1,2)}{\circ}\ar@{.}[urr]&{}                                &&\circ &{}
}.
$
\normalsize

They all produce isomorphic algebras.
Other $2$-slice quivers can be obtained by taking $\tau$-mutations on homogeneous ones with $S_4 =s^{(2,0)}S_1$, $S_5=s^{(1,0)}S_2$ and $S_5=s^{(1,1)}s^{(2,0)}S_1$.

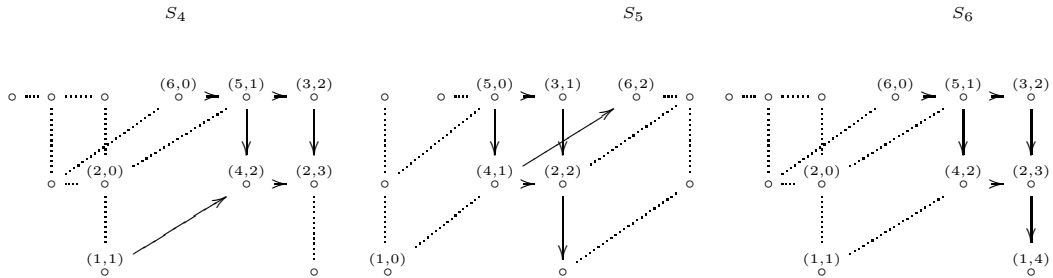
\tiny
$
\xymatrix@C=0.15cm@R0.6cm{&&&S_4&&& &{}{} &&&S_5&& &{}{}{} &&&S_6&&&\\
 \circ \ar@{.}[r] &\circ \ar@{.}[r]\ar@{.}[d] &\circ\ar@{.}[d] &{}   \stackrel{(6,0)}{\circ} \ar[r] &\stackrel{(5,1)}{\circ} \ar[r]\ar[d] & \stackrel{(3,2)}{\circ}\ar[d] &{}{}\circ\ar@{.}[d] &{}   {\circ} \ar@{.}[r] &\stackrel{(5,0)}{\circ} \ar[r]\ar[d] & \stackrel{(3,1)}{\circ}\ar[d] &{} \stackrel{(6,2)}{\circ} \ar@{.}[r] &\circ\ar@{.}[d] &{}{}{} {\circ} \ar@{.}[r] &{\circ} \ar@{.}[r]\ar@{.}[d] & {\circ}\ar@{.}[d] &{} \stackrel{(6,0)}{\circ} \ar[r] & \stackrel{(5,1)}{\circ} \ar[r]\ar[d] &\stackrel{(3,2)}{\circ}\ar[d] &{}\\
 &\circ \ar@{.}[r] \ar@{.}[urr]      &\stackrel{(2,0)}{\circ}\ar@{.}[d]\ar@{.}[urr] &{}              &\stackrel{(4,2)}{\circ}  \ar[r]    &\stackrel{(2,3)}{\circ}\ar@{.}[d] &{} {}  {\circ}\ar@{.}[d]\ar@{.}[urr] &{}              &\stackrel{(4,1)}{\circ}\ar[r]\ar[urr]       &\stackrel{(2,2)}{\circ}\ar[d]\ar@{.}[urr] &{}              &\circ  &{}{}{}   &{\circ}\ar@{.}[r]\ar@{.}[urr]       &\stackrel{(2,0)}{\circ}\ar@{.}[d]\ar@{.}[urr] &{}              &\stackrel{(4,2)}{\circ} \ar[r]       &\stackrel{(2,3)}{\circ}\ar[d] &{}       \\
 &&\stackrel{(1,1)}{\circ}\ar[urr] &{}                                 &&\circ &{}{}\stackrel{(1,0)}{\circ}\ar@{.}[urr] &{}                                 &&\circ\ar@{.}[urr] &{}                                 & &{}{}{}&&\stackrel{(1,1)}{\circ}\ar@{.}[urr]&{}                                &&\stackrel{(1,4)}{\circ} &{}
}.
$

\normalsize

The $\tau_{\perp}$-hammocks are as follows.
$$
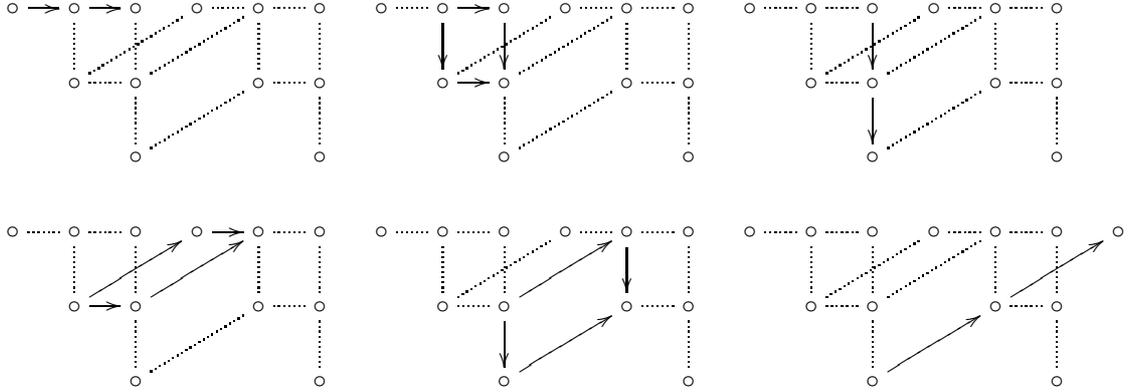
\xymatrix@C=0.4cm@R0.6cm{
\circ \ar[r] &\circ \ar[r]\ar@{.}[d] &\circ\ar@{.}[d] &{}  \circ \ar@{.}[r] &\circ \ar@{.}[r]\ar@{.}[d] &\circ\ar@{.}[d] &{}{} \circ \ar@{.}[r] &\circ \ar[r]\ar[d] &\circ\ar[d] &{} \circ \ar@{.}[r] &\circ \ar@{.}[r]\ar@{.}[d] &\circ\ar@{.}[d] &{}{} \circ \ar@{.}[r] &\circ \ar@{.}[r]\ar@{.}[d] &\circ\ar[d] &{} \circ \ar@{.}[r] &\circ \ar@{.}[r]\ar@{.}[d] &\circ\ar@{.}[d]&&\\
             &\circ \ar@{.}[r] \ar@{.}[urr]      &\circ\ar@{.}[d]\ar@{.}[urr] &{}              &\circ \ar@{.}[r]     &\circ\ar@{.}[d] &{}{}              &\circ \ar[r] \ar@{.}[urr]      &\circ\ar@{.}[d]\ar@{.}[urr] &{}              &\circ \ar@{.}[r]      &\circ\ar@{.}[d] &{}{}              &\circ \ar@{.}[r] \ar@{.}[urr]      &\circ\ar[d]\ar@{.}[urr] &{}              &\circ \ar@{.}[r]      &\circ\ar@{.}[d] &&\\
                                &&\circ\ar@{.}[urr] &{}                                 &&\circ &{}{}                                 &&\circ\ar@{.}[urr] &{}                                 &&\circ &{}{}                                 &&\circ\ar@{.}[urr] &{}                                 &&\circ&&\\%first
\circ \ar@{.}[r] &\circ \ar@{.}[r]\ar@{.}[d] &\circ\ar@{.}[d] &{}  \circ \ar[r] &\circ \ar@{.}[r]\ar@{.}[d] &\circ\ar@{.}[d]  &{}{} \circ \ar@{.}[r] &\circ \ar@{.}[r]\ar@{.}[d] &\circ\ar@{.}[d] &{} \circ \ar@{.}[r] &\circ \ar@{.}[r]\ar[d] &\circ\ar@{.}[d] &{} \circ \ar@{.}[r] &\circ \ar@{.}[r]\ar@{.}[d] &\circ\ar@{.}[d] &{} \circ \ar@{.}[r] &\circ \ar@{.}[r]\ar@{.}[d] &\circ\ar@{.}[d]&\circ& \\
             &\circ \ar[r] \ar[urr]      &\circ\ar@{.}[d]\ar[urr] &{}              &\circ \ar@{.}[r]     &\circ\ar@{.}[d] &{}{}              &\circ \ar@{.}[r] \ar@{.}[urr]      &\circ\ar[d]\ar[urr] &{}              &\circ \ar@{.}[r]      &\circ\ar@{.}[d] &{}{}              &\circ \ar@{.}[r] \ar@{.}[urr]      &\circ\ar@{.}[d]\ar@{.}[urr] &{}              &\circ \ar@{.}[r]\ar[urr]      &\circ\ar@{.}[d] &&\\
                                &&\circ\ar@{.}[urr] &{}                                 &&\circ &{}{}                                 &&\circ\ar[urr] &{}                                 &&\circ &{}{}                                 &&\circ\ar[urr] &{}                                 &&\circ&&\\%second
}.
$$

The dual one is a $1$-slice algebra of finite type.
Here are the  homogeneous $1$-slice quivers of finite type:

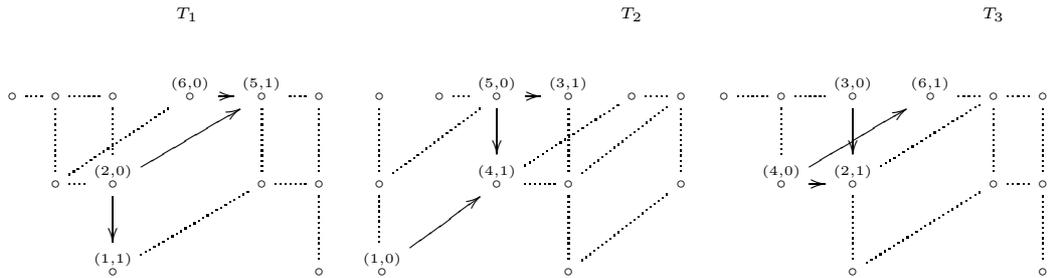
\tiny
$
\xymatrix@C=0.2cm@R0.6cm{&&&T_1&&& &{}{} &&&T_2&& &{}{}{} &&&T_3&&&\\
 \circ \ar@{.}[r] &\circ \ar@{.}[r]\ar@{.}[d] &\circ\ar@{.}[d] &{}   \stackrel{(6,0)}{\circ} \ar[r] &\stackrel{(5,1)}{\circ} \ar@{.}[r]\ar@{.}[d] & {\circ}\ar@{.}[d] &{}{}\circ\ar@{.}[d] &{}   {\circ} \ar@{.}[r] &\stackrel{(5,0)}{\circ} \ar[r]\ar[d] & \stackrel{(3,1)}{\circ}\ar@{.}[d] &{} {\circ} \ar@{.}[r] &\circ\ar@{.}[d] &{}{}{} {\circ} \ar@{.}[r] &{\circ} \ar@{.}[r]\ar@{.}[d] & \stackrel{(3,0)}{\circ}\ar[d] &{} \stackrel{(6,1)}{\circ} \ar@{.}[r] & {\circ} \ar@{.}[r]\ar@{.}[d] &\circ\ar@{.}[d] &{}\\
 &\circ \ar@{.}[r] \ar@{.}[urr]      &\stackrel{(2,0)}{\circ}\ar[d]\ar[urr] &{}              &{\circ}  \ar@{.}[r]    &\circ\ar@{.}[d] &{} {}  {\circ}\ar@{.}[d]\ar@{.}[urr] &{}              &\stackrel{(4,1)}{\circ}\ar@{.}[r]\ar@{.}[urr]       &{\circ}\ar@{.}[d]\ar@{.}[urr] &{}              &\circ  &{}{}{}   &\stackrel{(4,0)}{\circ}\ar[r]\ar[urr]       &\stackrel{(2,1)}{\circ}\ar@{.}[d]\ar@{.}[urr] &{}              &\circ \ar@{.}[r]       &\circ\ar@{.}[d] &{}       \\
 &&\stackrel{(1,1)}{\circ}\ar@{.}[urr] &{}                                 &&\circ &{}{}\stackrel{(1,0)}{\circ}\ar[urr] &{}                                 &&\circ\ar@{.}[urr] &{}                                 & &{}{}{}&&{\circ}\ar@{.}[urr]&{}                                &&\circ &{}
}.
$
\normalsize
They are isomorphic $1$-slices in $\olQ$.
Using $\tau_{\perp}$-mutation, we get all the $1$-slice quivers in $\olQ$.
The non-isomorphic ones are $T_4= s^{(6,0)}T_1$,  $T_5= s^{(2,0)}T_1$ and $T_6= s^{(1,1)}s^{(2,0)}T_1$.

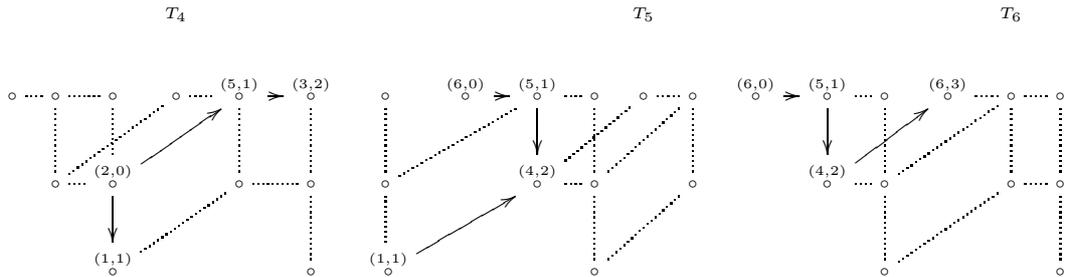
\tiny
$
\xymatrix@C=0.2cm@R0.6cm{&&&T_4&&& &{}{} &&&T_5&& &{}{}{} &&&T_6&&&\\
 \circ \ar@{.}[r] &\circ \ar@{.}[r]\ar@{.}[d] &\circ\ar@{.}[d] &{}   {\circ} \ar@{.}[r] &\stackrel{(5,1)}{\circ} \ar[r]\ar@{.}[d] & \stackrel{(3,2)}{\circ}\ar@{.}[d] &{}{}\circ\ar@{.}[d] &{}   \stackrel{(6,0)}{\circ} \ar[r] &\stackrel{(5,1)}{\circ} \ar@{.}[r]\ar[d] & {\circ}\ar@{.}[d] &{} {\circ} \ar@{.}[r] &\circ\ar@{.}[d] &{}{}{} \stackrel{(6,0)}{\circ} \ar[r] &\stackrel{(5,1)}{\circ} \ar@{.}[r]\ar[d] & {\circ}\ar@{.}[d] &{} \stackrel{(6,3)}{\circ} \ar@{.}[r] & {\circ} \ar@{.}[r]\ar@{.}[d] &\circ\ar@{.}[d] &{}\\
 &\circ \ar@{.}[r] \ar@{.}[urr]      &\stackrel{(2,0)}{\circ}\ar[d]\ar[urr] &{}              &{\circ}  \ar@{.}[r]    &\circ\ar@{.}[d] &{} {}  {\circ}\ar@{.}[d]\ar@{.}[urr] &{}              &\stackrel{(4,2)}{\circ}\ar@{.}[r]\ar@{.}[urr]       &{\circ}\ar@{.}[d]\ar@{.}[urr] &{}              &\circ  &{}{}{}   &\stackrel{(4,2)}{\circ}\ar@{.}[r]\ar[urr]       &{\circ}\ar@{.}[d]\ar@{.}[urr] &{}              &\circ \ar@{.}[r]       &\circ\ar@{.}[d] &{}       \\
 &&\stackrel{(1,1)}{\circ}\ar@{.}[urr] &{}                                 &&\circ &{}{}\stackrel{(1,1)}{\circ}\ar[urr] &{}                                 &&\circ\ar@{.}[urr] &{}                                 & &{}{}{}&&{\circ}\ar@{.}[urr]&{}                                &&\circ &{}
}.
$
\normalsize

Double slices are as following.

$$ D(S_1)
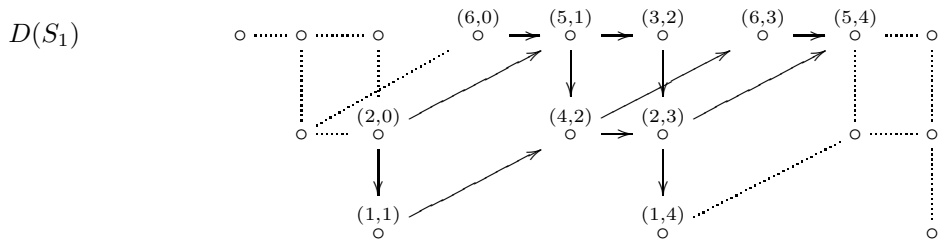
\xymatrix@C=0.4cm@R0.6cm{
&&  {}{}  & {\circ} \ar@{.}[r] &{\circ} \ar@{.}[r]\ar@{.}[d] &{\circ}\ar@{.}[d] &{}  \stackrel{(6,0)}{\circ} \ar[r] & \stackrel{(5,1)}{\circ} \ar[r]\ar[d] &\stackrel{(3,2)}{\circ}\ar@[green][d] &{} \stackrel{(6,3)}{\circ} \ar@[red][r] &\stackrel{(5,4)}{\circ} \ar@{.}[r]\ar@{.}[d] &{\circ}\ar@{.}[d] &{} \\ %{\circ} \ar[r] &{\circ} \ar[r]\ar[d] &{\circ}\ar[d] &  {}{}  &&\\
&&   {}{}    &           &{\circ} \ar@{.}[r] \ar@{.}[urr]      & \stackrel{(2,0)}{\circ}\ar[d]\ar[urr] &{}              & \stackrel{(4,2)}{\circ} \ar@[green][r]\ar@[green][urr]       &\stackrel{(2,3)}{\circ}\ar@[red][d]\ar@[red][urr] &{}             &{\circ} \ar@{.}[r] %\ar[urr]
    &{\circ}\ar@{.}[d] \\ %\ar[urr] &{} \\%             &{\circ} \ar[r]      &{\circ}\ar[d] &  {}{}  &&\\
&&   {}{}   &                               &&\stackrel{(1,1)}{\circ}\ar[urr] &{}                                 &&\stackrel{(1,4)}{\circ}\ar@{.}[urr] &{}                                 &&{\circ} \\ %\ar[urr] &{}               %                  &&{\circ} &  {}{}  &&\\
}.
$$

$$D(S_2)
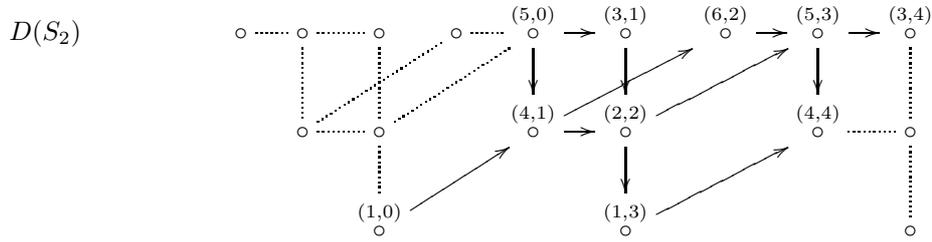
\xymatrix@C=0.4cm@R0.6cm{
&&  {}{}  & {\circ} \ar@{.}[r] &{\circ} \ar@{.}[r]\ar@{.}[d] &{\circ}\ar@{.}[d] &{}  {\circ} \ar@{.}[r] &\stackrel{(5,0)}{\circ} \ar[r]\ar[d] &\stackrel{(3,1)}{\circ}\ar[d] &{} \stackrel{(6,2)}{\circ} \ar@[green][r] &\stackrel{(5,3)}{\circ} \ar@[red][r]\ar@[red][d] &\stackrel{(3,4)}{\circ}\ar@{.}[d] &{} \\% {\circ} \ar[r] &{\circ} \ar[r]\ar[d] &{\circ}\ar[d] &  {}{}  &&\\
&&   {}{}    &           &{\circ} \ar@{.}[r] \ar@{.}[urr]      &{\circ}\ar@{.}[d]\ar@{.}[urr] &{}              & \stackrel{(4,1)}{\circ} \ar[r]\ar[urr]       &\stackrel{(2,2)}{\circ}\ar@[green][d]\ar@[green][urr] &{}              &\stackrel{(4,4)}{\circ} \ar@{.}[r] %\ar[urr]
      &{\circ}\ar@{.}[d] \\ % \ar@{.}[urr] &{}              &{\circ} \ar[r]      &{\circ}\ar[d] &  {}{}  &&\\
&&   {}{}   &                               &&\stackrel{(1,0)}{\circ}\ar[urr] &{}                                 &&\stackrel{(1,3)}{\circ}\ar@[red][urr] &{}                                 &&{\circ} \\ %\ar[urr] &{}                                 &&{\circ} &  {}{}  &&\\
}.
$$
$$D(S_3)
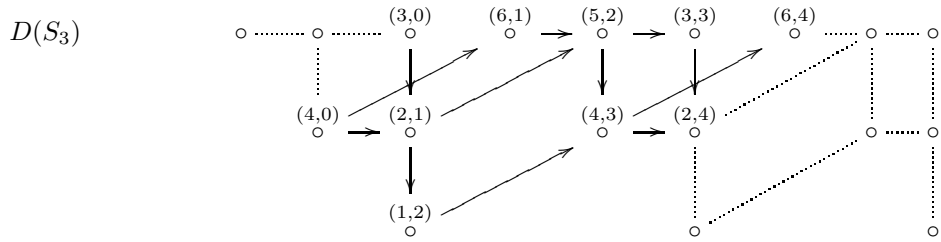
\xymatrix@C=0.4cm@R0.6cm{
&&  {}{}  & {\circ} \ar@{.}[r] &{\circ} \ar@{.}[r]\ar@{.}[d] &\stackrel{(3,0)}{\circ}\ar[d] &{}  \stackrel{(6,1)}{\circ} \ar[r] &\stackrel{(5,2)}{\circ} \ar@[green][r]\ar@[green][d] &\stackrel{(3,3)}{\circ}\ar@[red][d] &{} \stackrel{(6,4)}{\circ} \ar@{.}[r] &{\circ} \ar@{.}[r]\ar@{.}[d] &{\circ}\ar@{.}[d] &{} \\ % {\circ} \ar[r] &{\circ} \ar[r]\ar[d] &{\circ}\ar[d] &  {}{}  &&\\
&&   {}{}    &           &\stackrel{(4,0)}{\circ} \ar[r] \ar[urr]      &\stackrel{(2,1)}{\circ}\ar[d]\ar[urr] &{}              &\stackrel{(4,3)}{\circ} \ar@[red][r]\ar@[red][urr]       &\stackrel{(2,4)}{\circ}\ar@{.}[d]\ar@{.}[urr]
&{}              &{\circ} \ar@{.}[r] %\ar@{.}[urr]
     &{\circ}\ar@{.}[d] \\% \ar@{.}[urr] &{}              &{\circ} \ar[r]      &{\circ}\ar[d] &  {}{}  &&\\
&&   {}{}   &                               &&\stackrel{(1,2)}{\circ}\ar@[green][urr] &{}                                 &&{\circ}\ar@{.}[urr] &{}                                 &&{\circ} \\ %\ar[urr] &{}                                 &&{\circ} &  {}{}  &&\\
}.
$$
$$D(S_4)
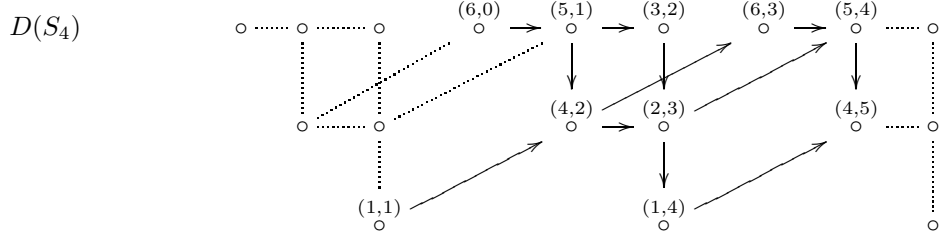
\xymatrix@C=0.4cm@R0.6cm{
&&  {}{}  & {\circ} \ar@{.}[r] &{\circ} \ar@{.}[r]\ar@{.}[d] &{\circ}\ar@{.}[d] &{}  \stackrel{(6,0)}{\circ} \ar[r] &\stackrel{(5,1)}{\circ} \ar[r]\ar[d] & \stackrel{(3,2)}{\circ}\ar[d] &{} \stackrel{(6,3)}{\circ}\ar@[red][r] &\stackrel{(5,4)}{\circ} \ar@{.}[r]\ar@[red][d] &{\circ}\ar@{.}[d] &{} \\ %{\circ} \ar[r] &{\circ} \ar[r]\ar[d] &{\circ}\ar[d] & {}{}  &&\\
&&   {}{}    &           &{\circ} \ar@{.}[r] \ar@{.}[urr]      &{\circ}\ar@{.}[d]\ar@{.}[urr] &{}              & \stackrel{(4,2)}{\circ} \ar[r]\ar@[green][urr]       &\stackrel{(2,3)}{\circ}\ar@[green][d]\ar@[green][urr] &{}              &\stackrel{(4,5)}{\circ} \ar@{.}[r] %\ar[urr]
     &{\circ} \ar@{.}[d] \\ %\ar@{.}[urr] &{}              &{\circ} \ar[r]      &{\circ}\ar[d] &  {}{}  &&\\
&&   {}{}   &                               &&\stackrel{(1,1)}{\circ}\ar[urr] &{}                                 &&\stackrel{(1,4)}{\circ}\ar@[red][urr] &{}                                 &&{\circ} \\%\ar[urr] &{}                                 &&{\circ} &  {}{}  &&\\
}.
$$
$$D(S_5)
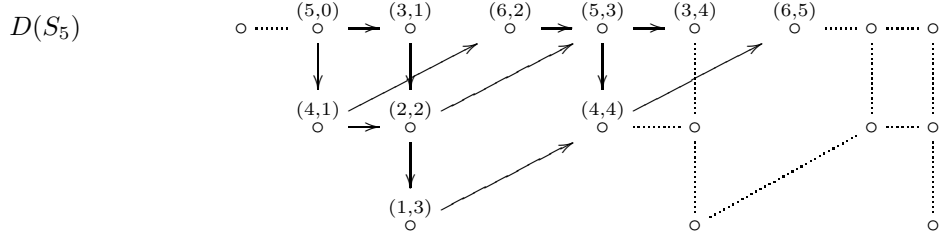
\xymatrix@C=0.4cm@R0.6cm{
&&  {}{}  & {\circ} \ar@{.}[r] &\stackrel{(5,0)}{\circ} \ar[r]\ar[d] &\stackrel{(3,1)}{\circ}\ar[d] &{}  \stackrel{(6,2)}{\circ} \ar@[green][r] &\stackrel{(5,3)}{\circ} \ar@[red][r]\ar@[red][d] &\stackrel{(3,4)}{\circ}\ar@{.}[d] &{} \stackrel{(6,5)}{\circ} \ar@{.}[r] &{\circ} \ar@{.}[r]\ar@{.}[d] &{\circ}\ar@{.}[d] &{} \\ %{\circ} \ar[r] &{\circ} \ar[r]\ar[d] &{\circ}\ar[d] &  {}{}  &&\\
&&   {}{}    &           &\stackrel{(4,1)}{\circ} \ar[r] \ar[urr]      &\stackrel{(2,2)}{\circ}\ar[d]\ar@[green][urr] &{}              &\stackrel{(4,4)}{\circ} \ar@{.}[r]\ar@[red][urr]       &{\circ}\ar@{.}[d]%\ar[urr]
&{}              &{\circ} \ar@{.}[r] %\ar[urr]
&{\circ}\ar@{.}[d] \\ %\ar[urr] &{}              &{\circ} \ar[r]      &{\circ}\ar[d] &  {}{}  &&\\
&&   {}{}   &                               &&\stackrel{(1,3)}{\circ}\ar@[green][urr] &{}                                 &&{\circ}\ar@{.}[urr] &{}                                 &&{\circ} \\%\ar[urr] &{}                                 &&{\circ} &  {}{}  &&\\
}.
$$
$$D(S_6)
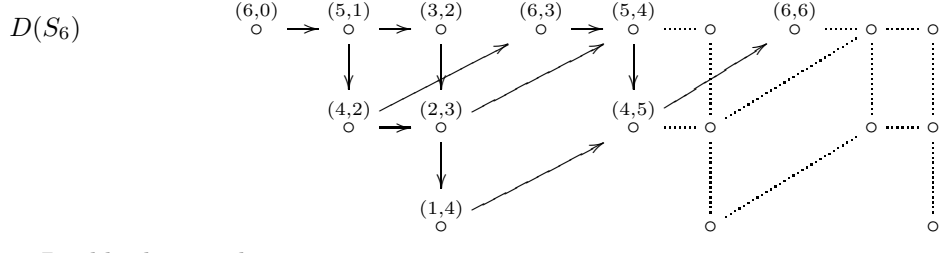
\xymatrix@C=0.4cm@R0.6cm{
&&  {}{}  & \stackrel{(6,0)}{\circ} \ar[r] &\stackrel{(5,1)}{\circ} \ar[r]\ar[d] & \stackrel{(3,2)}{\circ}\ar[d] &{}  \stackrel{(6,3)}{\circ} \ar@[red][r] &\stackrel{(5,4)}{\circ} \ar@{.}[r]\ar@[red][d] &{\circ}\ar@{.}[d] &{} \stackrel{(6,6)}{\circ} \ar@{.}[r] &{\circ} \ar@{.}[r]\ar@{.}[d] &{\circ}\ar@{.}[d] &{}\\ % {\circ} \ar[r] &{\circ} \ar[r]\ar[d] &{\circ}\ar[d] &  {}{}  &&\\
&&   {}{}    &           &\stackrel{(4,2)}{\circ} \ar[r] \ar@[green][urr]      &\stackrel{(2,3)}{\circ}\ar[d]\ar@[green][urr] &{}              &\stackrel{(4,5)}{\circ} \ar@{.}[r]\ar@[red][urr]       &{\circ}\ar@{.}[d] \ar@{.}[urr] &{}              &{\circ} \ar@{.}[r] %\ar[urr]
  &{\circ}\ar@{.}[d]  \\ %\ar[urr] &{}              &{\circ} \ar[r]      &{\circ}\ar[d] &  {}{}  &&\\
&&   {}{}   &                               &&\stackrel{(1,4)}{\circ}\ar@[green][urr] &{}                                 &&{\circ}\ar@{.}[urr] &{}                                 &&{\circ} \\ %\ar[urr] &{}                                 &&{\circ} &  {}{}  &&\\
}.
$$

Double slices with respect to $\tau_{\perp}$.

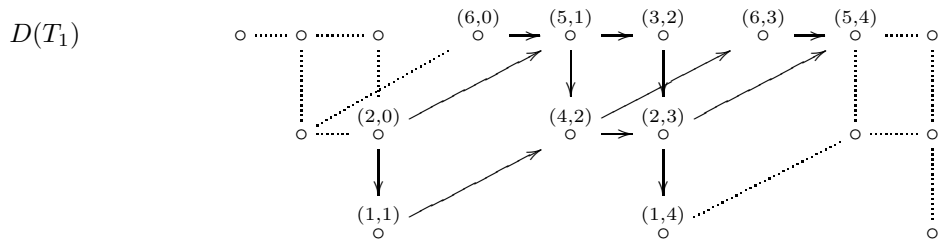
$$D(T_1)
\xymatrix@C=0.4cm@R0.6cm{
&&  {}{}  & {\circ} \ar@{.}[r] &{\circ} \ar@{.}[r]\ar@{.}[d] &{\circ}\ar@{.}[d] &{}  \stackrel{(6,0)}{\circ} \ar[r] &\stackrel{(5,1)}{\circ} \ar@[green][r]\ar@[green][d] &\stackrel{(3,2)}{\circ}\ar@[purple][d] &{} \stackrel{(6,3)}{\circ} \ar@[purple][r] &\stackrel{(5,4)}{\circ} \ar@{.}[r]\ar@{.}[d] &{\circ}\ar@{.}[d] &{} \\ %{\circ} \ar[r] &{\circ} \ar[r]\ar[d] &{\circ}\ar[d] &  {}{}  &&\\
&&   {}{}    &           &{\circ} \ar@{.}[r] \ar@{.}[urr]      &\stackrel{(2,0)}{\circ}\ar[d]\ar[urr] &{}              &\stackrel{(4,2)}{\circ} \ar@[purple][r]\ar@[purple][urr]       &\stackrel{(2,3)}{\circ}\ar@[purple][d]\ar@[purple][urr] &{}              &{\circ} \ar@{.}[r] %\ar@{.}[urr]
   &{\circ}\ar@{.}[d] \\ %\ar[urr] &{}              &{\circ} \ar[r]      &{\circ}\ar[d] &  {}{}  &&\\
&&   {}{}   &                               &&\stackrel{(1,1)}{\circ}\ar@[green][urr] &{}                                 &&\stackrel{(1,4)}{\circ}\ar@{.}[urr] &{}                                 &&{\circ} \\ %\ar[urr] &{}                                 &&{\circ} &  {}{}  &&\\
}.
$$

$$D(T_2)
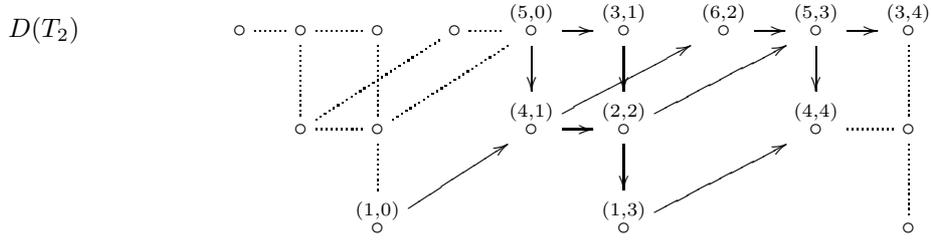
\xymatrix@C=0.4cm@R0.6cm{
&&  {}{}  & {\circ} \ar@{.}[r] &{\circ} \ar@{.}[r]\ar@{.}[d] &{\circ}\ar@{.}[d] &{}  {\circ} \ar@{.}[r] &\stackrel{(5,0)}{\circ} \ar[r]\ar[d] &\stackrel{(3,1)}{\circ}\ar@[green][d] &{} \stackrel{(6,2)}{\circ} \ar@[purple][r] &\stackrel{(5,3)}{\circ} \ar@[purple][r]\ar@[purple][d] &\stackrel{(3,4)}{\circ}\ar@{.}[d] &{} \\ %{\circ} \ar[r] &{\circ} \ar[r]\ar[d] &{\circ}\ar[d] &  {}{}  &&\\
&&   {}{}    &           &{\circ} \ar@{.}[r] \ar@{.}[urr]      &{\circ}\ar@{.}[d]\ar@{.}[urr] &{}              &\stackrel{(4,1)}{\circ} \ar@[green][r]\ar@[green][urr]       &\stackrel{(2,2)}{\circ}\ar@[purple][d]\ar@[purple][urr] &{}              &\stackrel{(4,4)}{\circ} \ar@{.}[r] %\ar@{.}[urr]
&{\circ}\ar@{.}[d] \\ %\ar@{.}[urr] &{}              &{\circ} \ar[r]      &{\circ}\ar[d] &  {}{}  &&\\
&&   {}{}   &                               &&\stackrel{(1,0)}{\circ}\ar[urr] &{}                                 &&\stackrel{(1,3)}{\circ} \ar@[purple][urr] &{}                                 &&{\circ} \\ %\ar[urr] &{}                                 &&{\circ} &  {}{}  &&\\
}.
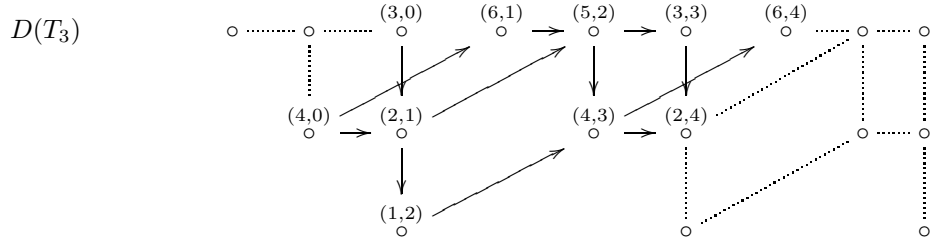
$$

$$
\xymatrix@C=0.4cm@R0.6cm{D(T_3)
&&  {}{}  & {\circ} \ar@{.}[r] &{\circ} \ar@{.}[r]\ar@{.}[d] &\stackrel{(3,0)}{\circ}\ar[d] &{}  \stackrel{(6,1)}{\circ} \ar@[green][r] &\stackrel{(5,2)}{\circ} \ar@[purple][r]\ar@[purple][d]&\stackrel{(3,3)}{\circ}\ar@[purple][d] &{} \stackrel{(6,4)}{\circ} \ar@{.}[r] &{\circ} \ar@{.}[r]\ar@{.}[d] &{\circ}\ar@{.}[d] &{} \\ % {\circ} \ar[r] &{\circ} \ar[r]\ar[d] &{\circ}\ar[d] &  {}{}  &&\\
&&   {}{}    &           &\stackrel{(4,0)}{\circ} \ar[r] \ar[urr]      &\stackrel{(2,1)}{\circ}\ar@[green][d]\ar@[green][urr] &{}              &\stackrel{(4,3)}{\circ} \ar@[purple][r]\ar@[purple][urr]       &\stackrel{(2,4)}{\circ}\ar@{.}[d]\ar@{.}[urr] &{}              &{\circ} \ar@{.}[r] %\ar[urr]
&{\circ}\ar@{.}[d] \\ %\ar[urr] &{}              &{\circ} \ar[r]      &{\circ}\ar[d] &  {}{}  &&\\
&&   {}{}   &                               &&\stackrel{(1,2)}{\circ}\ar@[purple][urr] &{}                                 &&{\circ}\ar@{.}[urr] &{}                                 &&{\circ} \\ %\ar[urr] &{}                                 &&{\circ} &  {}{}  &&\\
}.
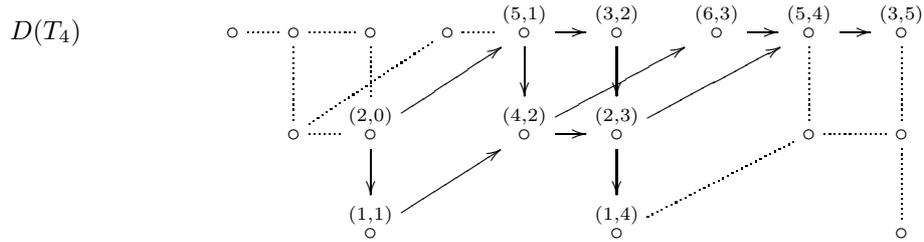
$$

$$
\xymatrix@C=0.4cm@R0.6cm{ D(T_4)
&&  {}{}  & {\circ} \ar@{.}[r] &{\circ} \ar@{.}[r]\ar@{.}[d] &{\circ}\ar@{.}[d] &{}  {\circ} \ar@{.}[r] &\stackrel{(5,1)}{\circ} \ar[r]\ar@[green][d] &\stackrel{(3,2)}{\circ}\ar@[green][d] &{} \stackrel{(6,3)}{\circ} \ar@[purple][r] &\stackrel{(5,4)}{\circ} \ar@[purple][r]\ar@{.}[d] &\stackrel{(3,5)}{\circ}\ar@{.}[d] &{}\\ % {\circ} \ar[r] &{\circ} \ar[r]\ar[d] &{\circ}\ar[d] &  {}{}  &&\\
&&   {}{}    &           &{\circ} \ar@{.}[r] \ar@{.}[urr]      &\stackrel{(2,0)}{\circ}\ar[d]\ar[urr] &{}              &\stackrel{(4,2)}{\circ} \ar@[purple][r]\ar@[purple][urr]       &\stackrel{(2,3)}{\circ}\ar@[purple][d]\ar@[purple][urr] &{}              &{\circ} \ar@{.}[r] %\ar[urr]
  &{\circ}\ar@{.}[d] \\ %\ar[urr] &{}              &{\circ} \ar[r]      &{\circ}\ar[d] &  {}{}  &&\\
&&   {}{}   &                               &&\stackrel{(1,1)}{\circ}\ar@[green][urr] &{}                                 &&\stackrel{(1,4)}{\circ}\ar@{.}[urr] &{}                                 &&{\circ} \\ %\ar[urr] &{}                                 &&{\circ} &  {}{}  &&\\
}.
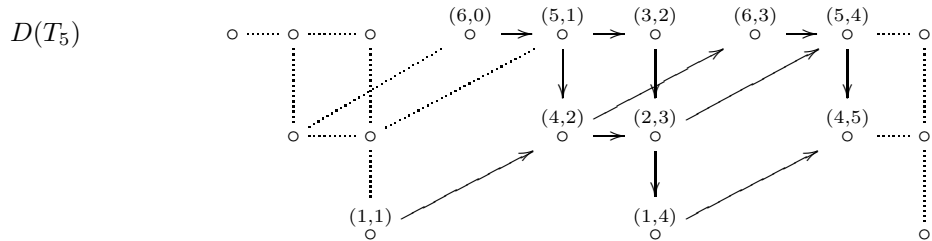
$$

$$
\xymatrix@C=0.4cm@R0.6cm{D(T_5)
&&  {}{}  & {\circ} \ar@{.}[r] &{\circ} \ar@{.}[r]\ar@{.}[d] &{\circ}\ar@{.}[d] &{}  \stackrel{(6,0)}{\circ} \ar[r] &\stackrel{(5,1)}{\circ} \ar@[green][r]\ar[d] &\stackrel{(3,2)}{\circ}\ar@[purple][d] &{} \stackrel{(6,3)}{\circ} \ar@[purple][r] &\stackrel{(5,4)}{\circ} \ar@{.}[r]\ar@[purple][d] &{\circ}\ar@{.}[d] &{} \\ % {\circ} \ar[r] &{\circ} \ar[r]\ar[d] &{\circ}\ar[d] &  {}{}  &&\\
&&   {}{}    &           &{\circ} \ar@{.}[r] \ar@{.}[urr]      &{\circ}\ar@{.}[d]\ar@{.}[urr] &{}              &\stackrel{(4,2)}{\circ} \ar@[green][r]\ar@[green][urr]       &\stackrel{(2,3)}{\circ}\ar@[purple][d]\ar@[purple][urr] &{}              &\stackrel{(4,5)}{\circ} \ar@{.}[r] %\ar[urr]
&{\circ}\ar@{.}[d] \\ %\ar[urr] &{}              &{\circ} \ar[r]      &{\circ}\ar[d] &  {}{}  &&\\
&&   {}{}   &                               &&\stackrel{(1,1)}{\circ}\ar[urr] &{}                                 &&\stackrel{(1,4)}{\circ}\ar@[purple][urr] &{}                                 &&{\circ} \\ %\ar[urr] &{}                                 &&{\circ} &  {}{}  &&\\
}.
$$

$$D(T_6)
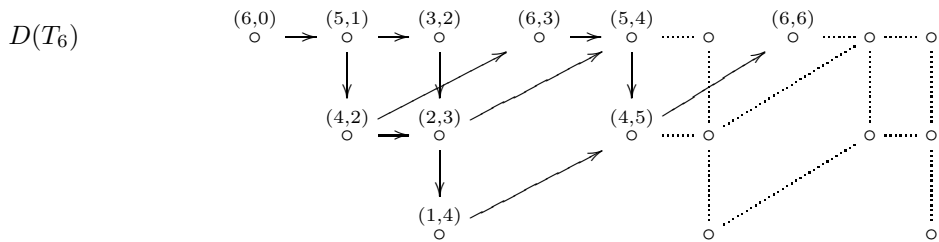
\xymatrix@C=0.4cm@R0.6cm{
&&  {}{}  & \stackrel{(6,0)}{\circ} \ar[r] &\stackrel{(5,1)}{\circ} \ar@[green][r]\ar[d] &\stackrel{(3,2)}{\circ}\ar@[purple][d] &{}  \stackrel{(6,3)}{\circ} \ar@[green][r] &\stackrel{(5,4)}{\circ} \ar@{.}[r]\ar@[purple][d] &{\circ}\ar@{.}[d] &{} \stackrel{(6,6)}{\circ} \ar@{.}[r] &{\circ} \ar@{.}[r]\ar@{.}[d] &{\circ}\ar@{.}[d] &{} \\ %{\circ} \ar[r] &{\circ} \ar[r]\ar[d] &{\circ}\ar[d] &  {}{}  &&\\
&&   {}{}    &           &\stackrel{(4,2)}{\circ} \ar@[green][r] \ar[urr]      &\stackrel{(2,3)}{\circ}\ar@[purple][d]\ar@[purple][urr] &{}              &\stackrel{(4,5)}{\circ} \ar@{.}[r]\ar@[purple][urr]       &{\circ}\ar@{.}[d] \ar@{.}[urr] &{}              &{\circ} \ar@{.}[r] %\ar[urr]
   &{\circ}\ar@{.}[d] \\ %\ar[urr] &{}              &{\circ} \ar[r]      &{\circ}\ar[d] &  {}{}  &&\\
&&   {}{}   &                               &&\stackrel{(1,4)}{\circ}\ar@[purple][urr] &{}                                 &&{\circ}\ar@{.}[urr] &{}                                 &&{\circ} \\ %\ar[urr] &{}                                 &&{\circ} &  {}{}  &&\\
}.
$$

}
\end{exa}

\section*{Acknowledgment}

The authors would like to thank the referees for reading the manuscript carefully and for suggestions and comments on revising and improving the paper.

{}

\end{document}